%% file: 0_main.tex
\documentclass{article}
\input{header}
\usepackage{notations}

\begin{document}

\date{}
\maketitle

\begin{abstract}
The finite element solution of Helmholtz problems near resonant or quasi-resonant frequencies poses significant challenges, as iterative solvers typically suffer from severely degraded convergence. We analyze the convergence behavior of GMRES applied to linear systems arising from such configurations. Theoretical convergence estimates are derived based on harmonic Ritz values, highlighting their proximity to small eigenvalues as a key determining factor. We further examine deflation strategies and their interplay with preconditioning techniques, using the Complex Shifted Laplacian preconditioner as a case study. Numerical experiments on resonant and quasi-resonant test cases validate the theoretical framework and demonstrate the effectiveness of deflation strategies. This study provides new insights and practical guidance for analyzing and improving iterative solvers for time-harmonic problems near resonances.
\end{abstract}


\input{1_intro.tex}
\input{2_gmres.tex}
\input{3_helmholtz.tex}
\input{4_experiment_cavity.tex}

\input{5_experiment_open.tex}
\input{6_conclusion.tex}

\section*{Acknowledgments}
We thank the reviewers for their constructive comments and suggestions that helped improve this manuscript.
We would also like to thank Zoïs Moitier for the insightful discussions and his valuable feedback.

\bibliographystyle{siam}
\bibliography{bibliography}

\end{document}

%% file: header.tex
\usepackage{amsfonts}
\usepackage{graphicx}
\usepackage{epstopdf}
\usepackage{algorithmic}
\ifpdf
  \DeclareGraphicsExtensions{.eps,.pdf,.png,.jpg}
\else
  \DeclareGraphicsExtensions{.eps}
\fi

\usepackage[a4paper, margin=2.5cm]{geometry}

\usepackage[bookmarks,
			plainpages=false,
			pdfpagelabels,
			colorlinks=true,
			linkcolor=orange,
			anchorcolor=gray,
			citecolor=gray,
			urlcolor=gray,
			hyperindex=true,
			breaklinks]
			{hyperref}

\usepackage{mathtools}
\usepackage{amsmath,amsfonts,amssymb,amsthm}
\usepackage{caption}
\usepackage{subcaption}
\usepackage{bm,bbm}
\usepackage{xcolor}
\usepackage{cleveref}

\newtheorem{theorem}{Theorem}[section]
\newtheorem{definition}[theorem]{Definition}
\newtheorem{remark}[theorem]{Remark}
\newtheorem{lemma}[theorem]{Lemma}
\newtheorem{proposition}[theorem]{Proposition}
\newtheorem{assumption}[theorem]{Assumption}

\newcommand{\funding}[1]{\textbf{\newline Funding:} #1}

\newcommand{\email}[1]{\href{mailto:#1}{#1}}

\title{Convergence analysis of GMRES applied to Helmholtz problems near resonances\thanks{Distributed under \href{https://creativecommons.org/licenses/by/4.0/}{Creative Commons CC BY 4.0} license.
\funding{This work was funded in part by the \textit{ANR JCJC project WavesDG} (research grant ANR-21-CE46-0010) and by the \textit{Agence de l'Innovation de Défense} [AID] through \textit{Centre Interdisciplinaire
d'Etudes pour la D\'efense et la S\'ecurit\'e} [CIEDS] (project 2022 ElectroMath).}}}

\author{V.~Dolean\thanks{
  Department of Mathematics and Computer Science, Eindhoven University of Technology, P.O. Box 513, 5600 MB Eindhoven, The Netherlands 
  (\email{v.dolean.maini@tue.nl}).}
\and P.~Marchand\thanks{POEMS, CNRS, Inria, ENSTA, Institut Polytechnique de Paris, 91120 Palaiseau, France 
(\email{pierre.marchand@inria.fr}, \email{axel.modave@ensta.fr}, \email{timothee.raynaud@ensta.fr}).}
\and A. Modave\footnotemark[3]
\and T. Raynaud\footnotemark[3]}

\newenvironment{keywords}{
  \par\vspace{1ex}
  \noindent\textbf{Keywords:}\hspace{0.5em}
  \begin{minipage}[t]{0.85\textwidth}
  \raggedright
}{
  \end{minipage}
  \par\vspace{1ex}
}

%% file: 1_intro.tex
\section{Introduction}\label{sec:intro}

The solution of time-harmonic wave propagation problems is of paramount importance in many fields of science and engineering, including acoustics, electromagnetism, and structural mechanics.
One of the most fundamental equation for modeling scalar waves is the Helmholtz equation,
\begin{align}
    \label{eq:Helmholtz}
        -\Delta u - k^2 u = f,
\end{align}
where \( k \in \mathbb{R}_{>0} \) is the wavenumber, $u$ is an unknown field, and \( f \) is a given source.
This equation is completed with boundary conditions.
Here, we are interested in the solution of~\eqref{eq:Helmholtz} when \( k^2 \) is close to an eigenvalue of the Laplacian.
Discretizing the resulting problem with the finite element method (FEM) leads to a linear system \(\mathbf{A} \mathbf{u} = \mathbf{b} \), where \( \mathbf{A} \in \mathbb{C}^{N \times N} \) is a sparse non-singular matrix which may be non-symmetric or non-Hermitian, \( \mathbf{u} \in \mathbb{C}^{N} \) is unknown, \( \mathbf{b} \in \mathbb{C}^{N} \) is given.
To ensure solution accuracy, the number of degrees of freedom per wavelength must be sufficiently high, as it must increase faster than the wavenumber to mitigate the pollution effect~\cite{BabuskaSauter1997IPE}.
For example, in dimension \(2\), the element diameter \(h\) should reduce faster than \(k^{-3/2}\) as the wavenumber \(k\) increases for piecewise linear elements (\(P_1\)).
Consequently, for large wavenumber problems, the resulting linear system can be extremely large.

Direct solvers are expensive and difficult to parallelize.
Solving finite element problems on structured grids yields an algorithmic complexity of \( \mathcal{O}(N^{3/2}) \) in 2D and \( \mathcal{O}(N^{2}) \) in 3D~\cite{DuffErismanEtAl2017DMS}.
Iterative methods, such as Krylov subspace methods, are more commonly used to solve large-scale problems.
These methods are easier to parallelize since they only require matrix-vector product computations.
The key challenge lies in controlling the number of iterations needed to achieve a given tolerance~\cite{Saad2003IMS}.

For non-normal matrices, the Generalized Minimal Residual (GMRES) method is a Krylov method of choice~\cite{SaadSchultz1986GGM}.
The GMRES convergence has been widely studied and several bounds have been established, for instance \cite{SimonciniSzyld2005OSC,VorstVuik1993SCB,SpillaneSzyld2024NCA,BeckermannGoreinovEtAl2005SRE,Embree2025EEB,CampbellIpsenEtAl1996GMP}.
However, there is no universally applicable estimate for all classes of matrices~\cite{Embree2022HDA}.
In particular, GMRES convergence in the case of discretized Helmholtz problems for different scenarios remains an open problem due to its highly indefinite or non-self-adjoint nature, as seen in previous studies~\cite{Erlangga2007AIM,ErnstGander2012WII,MarchandGalkowskiEtAl2022AGH}.
In addition, when the wavenumber is close to a resonance value, the number of GMRES iterations required to achieve a given tolerance can significantly increase.
Efficient acceleration techniques tailored for this class of problems must be used~\cite{Saad2003IMS,Vuik2018KSS}.
This topic is currently the focus of extensive research within the community. Acceleration techniques have been developed specifically for Helmholtz problems and more generally for non-normal cases, for example Complex Shifted Laplacian (CSL) preconditioning strategies~\cite{ErlanggaOosterleeEtAl2006NMB,ErlanggaVuikEtAl2006CMI,CoolsVanroose2013LFA,GanderGrahamEtAl2015AGH,RamosNabben2020TLS,RamosSeteEtAl2021PHE}, domain decomposition methods~\cite{DoleanJolivetEtAl2015IDD,DoleanFryEtAl2025AWR}, and deflation techniques~\cite{GarciaRamosKehlEtAl2020PDM,DwarkaVuik2022SML,SpillaneSzyld2024NCA}.

In this work, we focus on GMRES convergence for Helmholtz problems close to resonances.
In a closed domain, a resonance is a (real) wavenumber $k$ such that $k^2$ is an eigenvalue of the Laplace operator. In an unbounded domain, resonances are complex, but some of them approach the real axis. Their real parts are called quasi-resonances, and are wavenumbers for which the norm of the inverse of the Helmholtz operator grows rapidly.
Previous studies~\cite{MarchandGalkowskiEtAl2022AGH} provided estimates on the increase in the number of GMRES iterations for a sequence of quasi-modes as the frequency increases.
Here, we aim at a new understanding of these convergence behaviors based on
GMRES approximation of the underlying the properties of the linear system.
We also provide suggestions for studying the robustness of iterative solvers in these cases.
The contribution of deflation, which can be combined with preconditioning techniques, is highlighted.
This is particularly interesting for cases involving quasi-modes and modes.

The main aim of this paper is to investigate the relationships between GMRES convergence rates and the properties of linear systems associated with Helmholtz problems near resonances.
More specifically our main contributions are:
\begin{itemize}
    \setlength{\itemsep}{0pt}
    \setlength{\parskip}{0pt}
    \setlength{\parsep}{0pt}
    \item to generalize a GMRES residual bound based on harmonic Ritz values;
    \item to show that the approximation of every eigenvalue or group of eigenvalues by harmonic Ritz values affects convergence;
    \item to apply deflation techniques with a suitable choice of eigenvectors adapted to resonance problems to accelerate GMRES convergence and combine them with a preconditioner;
    \item to understand the numerical behavior of GMRES for different scenarios and acceleration methods.
\end{itemize}
The remainder of the paper is structured as follows. \Cref{sec:gmres} introduces the GMRES method and reviews key convergence properties. We then extend a GMRES residual bound based on harmonic Ritz values, showing that when these values approach a subset of eigenvalues closely enough, the corresponding eigenvalues cease to hinder convergence. In \Cref{sec:helmholtz}, we present the continuous and discretized Helmholtz problem and the settings in which resonances and quasi-resonances arise.
To efficiently solve the problem for cases near to resonances or quasi-resonances, we study and highlight the utility of a deflation technique, and we investigate its combination with a preconditioning strategy.
For this analysis, we selected the  Complex Shifted Laplacian (CSL) preconditioner with an incomplete LU (ILU) factorization as a case of study.
 Finally, in \Cref{sec:cavityProblem,sec:scatteringProblems}, we investigate numerical benchmarks near resonant regimes. These cases illustrate how the proposed tools influence GMRES convergence and allow us to assess the effectiveness and robustness of the acceleration techniques.

%% file: 2_gmres.tex
\section{GMRES}\label{sec:gmres}

The Generalized Minimal Residual (GMRES) method~\cite{SaadSchultz1986GGM} is a popular iterative algorithm for solving general linear systems.
In this section, we introduce its main components and discuss a standard convergence result.
We then propose new convergence results based on harmonic Ritz (HR) values, which allow a better interpretation of the GMRES residual evolution.


\subsection{Key ideas}\label{subsec:GMRESConvergence}

The GMRES method belongs to the class of Krylov subspace iterative methods, leveraging the Krylov subspace defined as
\begin{equation*}
    \mathcal{K}_l(\mathbf{A}, \mathbf{r}_0) := \mathrm{span}\{\mathbf{r}_0, \mathbf{A}\mathbf{r}_0, \mathbf{A}^2\mathbf{r}_0, \ldots, \mathbf{A}^{l-1}\mathbf{r}_0\},
    \quad \text{for $l<N$},
\end{equation*}
where $\mathbf{r}_0 := \mathbf{b} - \mathbf{A}\mathbf{x}_0$ is the initial residual and $\mathbf{x}_0$ is an initial guess.
At each iteration $l < N$, the approximate solution $\mathbf{x}_l$ belongs to $\mathbf{x}_0 + \mathcal{K}_l(\mathbf{A}, \mathbf{r}_0)$, and it is computed such that the residual $\mathbf{r}_l = \mathbf{b} - \mathbf{A}\mathbf{x}_l$ has minimum norm, i.e.
\begin{equation}
    \mathbf{x}_l := \argmin_{\mathbf{x} \in \mathbf{x}_0 + \mathcal{K}_l(\mathbf{A}, \mathbf{r}_0)} \|\mathbf{b} - \mathbf{A}\mathbf{x}\|_2.
    \label{eq:minResidual0}
\end{equation}
A critical component of GMRES is the Arnoldi process, which generates an orthonormal basis for the Krylov subspace.
The basis vectors are the columns of a matrix $\mathbf{U}_l\in\mathbb{C}^{N\times l}$.
This basis is used to compute the restriction of $\mathbf{A}$ to the Krylov subspace, $\mathbf{H}_l = \mathbf{U}_l^*\mathbf{A}\mathbf{U}_l \in \mathbb{C}^{l\times l}$, which is an upper Hessenberg matrix.
Each iteration of the GMRES algorithm can be summarized as follows: perform the Arnoldi process to compute $\mathbf{U}_l$ and $\mathbf{H}_l$, then solve the least-squares problem $\min_{\mathbf{y}} \|\beta \mathbf{e}_1 - \mathbf{H}_l\mathbf{y}\|_2$ with $\beta = \|\mathbf{r}_0\|_2$ and $\mathbf{e}_1$ the first canonical vector, and finally form the approximate solution $\mathbf{x}_l = \mathbf{x}_0 + \mathbf{U}_l\mathbf{y}$, see e.g.~\cite{Saad2003IMS}.


\subsection{Convergence results}
Several convergence bounds have been developed to analyze the behavior of GMRES, often by expressing the residual vector in terms of polynomials; see, for example,~\cite{Embree2022HDA}. The residual at iteration~$l$ satisfies the following relation:
\begin{equation}
    \label{eq:minResidual}
    \twonorm{\mathbf{r}_l} = \min_{q_l \in \mathcal{P}_l^1} \twonorm{q_l(\mathbf{A}) \mathbf{r}_0},
\end{equation}
where $\mathcal{P}_l^1$ is the set of polynomials of degree at most~$l$ with the constraint $q_l(0) = 1$. Throughout this paper, we denote by $p_l \in \mathcal{P}_l^1$ the polynomial that achieves this minimum, such that
\begin{equation}\label{eq:pl_def}
    \mathbf{r}_l = p_l(\mathbf{A}) \mathbf{r}_0.
\end{equation}

GMRES convergence can be understood in terms of the eigenvalue distribution of~$\mathbf{A}$ and the conditioning of its eigenvectors. If $\mathbf{A}$ is diagonalizable, we can write $\mathbf{A} = \mathbf{V} \mathbf{\Lambda} \mathbf{V}^{-1}$. For any polynomial $q_l \in \mathcal{P}_l^1$, it follows that:
\begin{equation*}
    \twonorm{q_l(\mathbf{A}) \mathbf{r}_0}
    = \twonorm{\mathbf{V} q_l(\mathbf{\Lambda}) \mathbf{V}^{-1} \mathbf{r}_0}
    \leq \kappa_2(\mathbf{V}) \twonorm{q_l(\mathbf{\Lambda})} \twonorm{\mathbf{r}_0},
\end{equation*}
where $\kappa_2(\mathbf{V}) := \twonorm{\mathbf{V}} \twonorm{\mathbf{V}^{-1}}$ is the condition number of the eigenvector matrix. This leads to the classical convergence bound~\cite{SaadSchultz1986GGM}:
\begin{equation}
    \label{eq:cvboundEV}
    \frac{\twonorm{\mathbf{r}_l}}{\twonorm{\mathbf{r}_0}} \leq \kappa_2(\mathbf{V}) \min_{q_l \in \mathcal{P}_l^1} \max_{\lambda \in \sigma(\mathbf{A})} \abs{q_l(\lambda)},
\end{equation}
where $\sigma(\mathbf{A})$ denotes the set of eigenvalues of $\mathbf{A}$. If $\mathbf{A}$ is normal, then it is possible to choose $\kappa_2(\mathbf{V}) = 1$, and the convergence depends solely on the eigenvalue distribution. In contrast, for non-normal matrices, $\kappa_2(\mathbf{V})$ can be large, making the bound in~\eqref{eq:cvboundEV} less predictive of actual convergence behavior~\cite{Embree2022HDA}.

Alternative convergence bounds relate to other properties of $\mathbf{A}$, such as its field of values~\cite{BeckermannGoreinovEtAl2005SRE,Embree2025EEB,CrouzeixPalencia2017NRI} or its pseudospectrum~\cite{TrefethenEmbree2005SPB}.
While these bounds have distinct advantages, none of them provides an accurate or systematic prediction of GMRES convergence.
Moreover, they describe global decay rates of the GMRES residual~\cite{Embree2022HDA}, by considering the norm of the residual at iteration \(l\) relative to the initial residual.
They do not capture specific and local convergence behaviors, such as stagnation phases followed by rapid decays.
Such patterns can be observed in the GMRES residual history, and can occur close to resonances (see \Cref{fig:cavityTrajRes} and \Cref{sec:cavityProblem,sec:scatteringProblems}).
To better understand this phenomenon, it is crucial to highlight the role of small eigenvalues related to these resonances.
This need motivates our interest in harmonic Ritz values.
These results also highlight how both spectral and non-normality properties of $\mathbf{A}$ shape the convergence of GMRES.\


\subsection{Harmonic Ritz values}\label{sec:harmonicRitzvalues}
Harmonic Ritz (HR) values provide estimates of the eigenvalues of $\mathbf{A}$ at a given iteration $l$ (see for example~\cite[Chapter 26]{TrefethenEmbree2005SPB}).
In the following, we assume that GMRES does not stagnate up to iteration $l$.
They are defined as follows:
\begin{definition}[Harmonic Ritz (HR) values]\label{def:harmonicRitzValues}
    At iteration $l > 0$, the \emph{harmonic Ritz values} ${\{\nu_j^{(l)}\}}_{j=1}^l$ are the roots of the minimizing polynomial $p_l \in \mathcal{P}_l^1$ from~\eqref{eq:minResidual}.
\end{definition}
By construction, the number of HR values increases with each GMRES iteration. Since $p_l(0) = 1$, none of the HR values can be zero. As a result, the polynomial $p_l$ can be factored as:
\begin{equation}
    \label{eq:minpol}
        p_l(z) = \prod_{j=1}^{l} \left(1 - \frac{z}{\nu_j^{(l)}}\right).
\end{equation}
The HR values satisfy a useful lower bound:
\begin{proposition}[Lower bound~{\cite[Theorem 4]{Cao1997NCB}}]\label{prop:boundHR}
    Assume $\mathbf{A}$ is non-singular.
    At iteration $l>0$, the HR values satisfy
    \begin{equation*}
        \abs{\nu_j^{(l)}} \geq s_{\min}(\mathbf{A}),
        \quad \forall j=1,\ldots,l,
    \end{equation*}
    where $s_{\min}(\mathbf{A})$ is the smallest singular value of $\mathbf{A}$.
\end{proposition}
As iterations progress, the HR values tend to approach the eigenvalues of $\mathbf{A}$ increasingly well. They converge to the full spectrum after $N$ iterations, or earlier in cases of ``lucky breakdown''~\cite[Chapter 26]{TrefethenEmbree2005SPB}. Proposition~\ref{prop:boundHR} guarantees that these approximations remain bounded away from zero as long as $\mathbf{A}$ is non-singular.

\begin{remark}[Ritz values]\label{rem:ritz_values}
The Ritz values are the eigenvalues of the upper Hessenberg matrix $\mathbf{H}_l$, and they coincide with the roots of the minimizing polynomial in the Full Orthogonalization Method (FOM). Like the HR values, Ritz values serve as approximations of the eigenvalues of $\mathbf{A}$. In~\cite{SluisVorst1986RCC}, Ritz values are used to analyze the convergence of the conjugate gradient method for symmetric positive definite systems. This approach is extended in~\cite{VorstVuik1993SCB} to study the convergence of GMRES for general non-normal matrices. However, because Ritz values are associated with FOM rather than GMRES, the analysis must relate the residual norms of the two methods.
As a result, the residual bounds derived from Ritz values involve a factor larger than $1$ that depends on the iteration~\cite[Theorems 2.5 and 2.6]{VorstVuik1993SCB}.
A more natural bound can be obtained using HR values~\cite{Cao1997NCB}, leading to a similar result as in~\cite{VorstVuik1993SCB}, but improved in that the factor no longer appears~\cite[Theorem 6]{Cao1997NCB}.
\end{remark}


\subsection{Convergence results with HR values}\label{sec:harmonicRitzGMRES}
For the reasons stated in Remark \ref{rem:ritz_values}, Cao~\cite{Cao1997NCB} proposed such a bound by linking GMRES convergence to the approximation of the smallest eigenvalue by an HR value. Building on this idea, we introduce two generalizations that consider the approximation of an entire subset of eigenvalues by corresponding HR values, rather than just a single one. This leads to a natural partition of the spectrum into two groups: the eigenvalues closely approximated by HR values, and the remaining ones. This strategy is reminiscent of spectral clustering techniques that separate eigenvalue distributions into dominant clusters and peripheral outliers~\cite{CampbellIpsenEtAl1996GMP,MarchandGalkowskiEtAl2022AGH}.

To formalize this spectral separation and derive quantitative bounds, we now consider the application of GMRES to a linear system of the form $\mathbf{A} \mathbf{x} = \mathbf{b}$, with $\mathbf{A} \in \mathbb{C}^{N \times N}$ and $\mathbf{b} \in \mathbb{C}^N$. We denote the residual vector at iteration $l$ by $\mathbf{r}_l$.
Ultimately, the matrix $\mathbf{A}$ will come from the finite elements discretization of a physical problem.
The convergence results presented in this section then rely on the following assumption.

\begin{assumption}\label{hyp:Adiag}
The matrix $\mathbf{A}$ is non-singular and diagonalizable.
\end{assumption}

We aim to analyze the residual evolution between two iterations, $l$ and $l + m$, assuming $0 < l < l + m < N$. We begin with a preliminary result that sets the stage for the generalized bounds.

\begin{lemma}\label{lem:convergenceBound}
Suppose that Assumption~\ref{hyp:Adiag} holds, and let $\mathbf{A} = \mathbf{V} \mathbf{\Lambda} \mathbf{V}^{-1}$ be an eigendecomposition of $\mathbf{A}$, with $\mathbf{\Lambda} = \mathrm{diag}(\lambda_1, \ldots, \lambda_N)$.
Let $J \in \{1, \ldots,l\}$, and consider a set $N_J^{(l)}$ of $J$ HR values at iteration $l$, along with a set $\Lambda_J$ of $J$ eigenvalues of $\mathbf{A}$, such that no HR value in $N_J^{(l)}$ is an eigenvalue in $\Lambda_J^c := \sigma(\mathbf{A}) \setminus \Lambda_J$.
Then, for any $m > 0$ such that $l + m < N$, the following bound holds:
\begin{equation*}
\frac{\twonorm{ \mathbf{r}_{l+m} }}{\twonorm{ \mathbf{r}_l }} \leq \min_{q_{m} \in \mathcal{P}_{m}^1} \twonorm{q_{m}(\mathbf{A})\, \mathbf{S}_{J}^{l}},
\end{equation*}
where
\begin{equation}
\mathbf{S}_J^l := \mathbf{V} \, \mathrm{diag}(s_1, \ldots, s_N) \, \mathbf{V}^{-1},
\quad \text{with} \quad s_i :=\begin{cases}
0 & \text{if } \lambda_i \in \Lambda_J, \\ s_J^l(\lambda_i) & \text{otherwise},
\end{cases}
\label{eq:SJl}
\end{equation}
and
\begin{equation}
s_J^{l}(z) := n_J(z) / d_J^l(z)
\label{eq:sJl}
\end{equation}
with
\begin{equation}
n_J(z) := \prod_{\lambda_j \in \Lambda_J} \left(1 - \frac{z}{\lambda_j}\right)
\quad \text{and} \quad
d_J^l(z) := \prod_{\nu_j^{(l)} \in N_J^{(l)}} \left(1 - \frac{z}{\nu_j^{(l)}}\right).
\end{equation}
\end{lemma}
\begin{proof}
Recall from~\eqref{eq:pl_def} that \( \mathbf{r}_l = p_l(\mathbf{A}) \mathbf{r}_0 \), where \( p_l \in \mathcal{P}_l^1 \) is the GMRES minimizing polynomial at iteration $l$.
Since the HR values in $N_J^{(l)}$ are roots of $p_l$~\eqref{eq:minpol}, one can write $p_l = d_J^l \, \hat{p}_{l-J}$, where $\hat{p}_{l-J}$ is a polynomial of degree $l - J$.
For any $ q_m \in \mathcal{P}_m^1$, the polynomial
\[
\tilde{p}_{l+m} := q_m \, n_J \, \hat{p}_{l-J}
\]
has degree at most \( l + m \) and satisfies \( \tilde{p}_{l+m}(0) = q_m(0) \, n_J(0) \, \hat{p}_{l-J}(0) = 1 \), since \( \hat{p}_{l-J}(0) = p_{l}(0)/d_{J}^l(0) = 1 \).
Therefore, \( \tilde{p}_{l+m} \in \mathcal{P}_{l+m}^1 \).
For any eigenvalue $\lambda_i$, one has
\begin{equation}\label{eq:ptildeLambda}
\tilde{p}_{l+m}(\lambda_i)
= q_m(\lambda_i) \, s_i \, p_l(\lambda_i).
\end{equation}
Indeed, if $\lambda_i \in \Lambda_J$, then one has $s_i = 0$ by definition and \(\tilde{p}_{l+m}(\lambda_i) = 0\) because $n_J(\lambda_i) = 0$.
If $\lambda_i \in \Lambda_J^c$, then one can write
\[
\tilde{p}_{l+m}(\lambda_i)
= q_m(\lambda_i) \, \frac{n_J(\lambda_i)}{d_J^l(\lambda_i)} \, d_J^l(\lambda_i) \, \hat{p}_{l-J}(\lambda_i)
= q_m(\lambda_i) \, s_J^l(\lambda_i) \, p_l(\lambda_i)
= q_m(\lambda_i) \, s_i \, p_l(\lambda_i)
\]
since $d_J^l(\lambda_i) \neq 0$ by assumption.
Applying \eqref{eq:ptildeLambda} to all the eigenvalues and using the basis of eigenvectors gives
\begin{equation}\label{eq:ptilde}
\tilde{p}_{l+m}(\mathbf{A}) = q_m(\mathbf{A}) \, \mathbf{S}_J^l \, p_l(\mathbf{A}),
\end{equation}
with $\mathbf{S}_J^l$ defined by~\eqref{eq:SJl}.
From the GMRES optimality property~\eqref{eq:minResidual} at iteration $l + m$, we have:
\[
\twonorm{ \mathbf{r}_{l+m} } = \min_{q_{l+m} \in \mathcal{P}_{l+m}^1} \twonorm{q_{l+m}(\mathbf{A}) \mathbf{r}_0}.
\]
In particular, using \( \tilde{p}_{l+m} \) yields:
\begin{align*}
\twonorm{ \mathbf{r}_{l+m} } &\leq \twonorm{\tilde{p}_{l+m}(\mathbf{A}) \mathbf{r}_0} \\
&= \twonorm{q_m(\mathbf{A}) \, \mathbf{S}_J^l \, p_l(\mathbf{A}) \mathbf{r}_0} \\
&= \twonorm{q_m(\mathbf{A}) \, \mathbf{S}_J^l \, \mathbf{r}_l} \\
&\leq \twonorm{q_m(\mathbf{A}) \, \mathbf{S}_J^l} \, \twonorm{\mathbf{r}_l}.
\end{align*}
Taking the minimum over all \( q_m \in \mathcal{P}_m^1 \) completes the proof.
\end{proof}
\Cref{lem:convergenceBound} provides a general residual estimate based on the positions of the HR values relative to the eigenvalues. To make this bound more explicit and interpretable, we now use the eigendecomposition of $\mathbf{A}$ to express the norm of \( q_m(\mathbf{A}) \mathbf{S}_J^l \) in spectral terms. This leads to our first main result, which quantifies the effect of the remaining eigenvalues—those not approximated by HR values—on the residual evolution.

\begin{theorem}[First residual estimate]\label{thm:harmonicRitzValues} If the assumptions of \Cref{lem:convergenceBound} hold, then, for iterations $l$ and $l+m$, we have
    \begin{align*}
        \frac{\twonorm{ \mathbf{r}_{l+m} }}{\twonorm{ \mathbf{r}_l }}  \leq \left( \sum_{\lambda_i\in \Lambda_J^c} \kappa(\lambda_i) \right) \ \max_{\lambda_i \in \Lambda_J^c} \abs{s_J^{l}(\lambda_i)} \ \min_{q_{m}\in \mathcal{P}_{m}^1} \max_{\lambda_i \in \Lambda_J^c} \abs{ q_{m}(\lambda_i) },
    \end{align*}
    with $\Lambda_J^c := \sigma(\mathbf{A}) \setminus \Lambda_J$ and $\kappa(\lambda_i) := \twonorm{\hat{\mathbf{v}}_i}\twonorm{\mathbf{v}_i}$, where $\hat{\mathbf{v}}_i$ and $\mathbf{v}_i \in \C^{N}$ are the left and right eigenvectors associated to $\lambda_i\in\sigma(\mathbf{A})$, with $\hat{\mathbf{v}}_i^{*}\mathbf{v}_i=1$.
\end{theorem}
\begin{proof}
Since $\mathbf{A}$ is diagonalizable, it admits the decomposition $\mathbf{A} = \mathbf{V} \mathbf{\Lambda} \mathbf{V}^{-1}$, where $\mathbf{V} \in \mathbb{C}^{N \times N}$ is the matrix of right eigenvectors, and $\mathbf{\Lambda} = \mathrm{diag}(\lambda_1, \ldots, \lambda_N)$ is the diagonal matrix of eigenvalues. We denote by $\mathbf{v}_i$ the $i$-th column of $\mathbf{V}$ (the right eigenvector associated with $\lambda_i$), and by $\hat{\mathbf{v}}_i^*$ the $i$-th row of $\mathbf{V}^{-1}$, corresponding to the conjugate transpose of the left eigenvector associated with $\lambda_i$ (see~\cite[Chapter 52]{TrefethenEmbree2005SPB}).

    By using the eigenvalue decomposition and \Cref{lem:convergenceBound}, we obtain
    \begin{align*}
        \frac{\twonorm{ \mathbf{r}_{l+m} }}{\twonorm{ \mathbf{r}_l }} & \leq
        \min_{q_{m} \in \mathcal{P}_{m}^1} \twonorm{\mathbf{V}q_{m}\left( \mathbf{\Lambda}\right) \mathrm{diag}(s_1, \ldots, s_N) \mathbf{V}^{-1}} \\
        & = \min_{q_{m} \in \mathcal{P}_{m}^1} \twonorm{\sum_{\lambda_{i} \in \sigma(\mathbf{A})} q_{m}(\lambda_{i}) s_i \mathbf{v}_i \hat{\mathbf{v}}_{i}^{*}}.
    \end{align*}
    Because $s_i = 0$ for all $\lambda_i \in \Lambda_J$ and $s_i = s_J^l(\lambda_i)$ for all $\lambda_i \in \Lambda_J^c$, by~\eqref{eq:SJl}, the sum can be restricted to $\lambda_{i} \in \Lambda_J^c$, leading to
    \begin{align*}
        \frac{\twonorm{ \mathbf{r}_{l+m} }}{\twonorm{ \mathbf{r}_l }}
        & \leq
        \min_{q_{m} \in \mathcal{P}_{m}^1} \twonorm{\sum_{\lambda_i \in \Lambda_J^c} q_{m}(\lambda_i) s_{J}^{l}(\lambda_i) \mathbf{v}_i \hat{\mathbf{v}}_{i}^{*}}
        \\
        & \leq \min_{q_{m} \in \mathcal{P}_{m}^1} \sum_{\lambda_i \in \Lambda_J^c} \abs{q_{m}(\lambda_i)} \abs{s_{J}^{l}(\lambda_i)} \twonorm{\mathbf{v}_i}\twonorm{\hat{\mathbf{v}}_{i}^{*}}
        \\
        & = \min_{q_{m} \in \mathcal{P}_{m}^1} \sum_{\lambda_i \in \Lambda_J^c} \abs{q_{m}(\lambda_i)} \abs{s_{J}^{l}(\lambda_i)} \kappa(\lambda_i)
        \\
        & \leq \left(\sum_{\lambda_i\in \Lambda_J^c} \kappa(\lambda_i)\right) \max_{\lambda_i \in \Lambda_J^c} \abs{s_{J}^{l}(\lambda_i)} \min_{q_{m}\in \mathcal{P}_{m}^1} \max_{\lambda_i \in \Lambda_J^c} \abs{ q_{m}(\lambda_i) }.
    \end{align*}
\end{proof}
\paragraph{Interpretation}
If, at iteration $l$, the HR values in $N_J^{(l)}$ closely approach the eigenvalues in $\Lambda_J$, then the ratio $\max_{\lambda_i \in \Lambda_J^c} \abs{s_J^{l}(\lambda_i)}$ approaches 1.
Moreover, these eigenvalues are no longer involved in the $\min \max$ problem, and their condition numbers are removed from the sum:
 the convergence rate for subsequent iterations is effectively governed only by the remaining eigenvalues in $\Lambda_J^c$. For highly non-normal problems, the influence of the eigenvalues in $\Lambda_J$ may be indirectly reflected in high values of the remaining $\kappa(\lambda_i)$, but in any case, the term related to non-normality decreases as the number of iterations increases.

An important advantage of expressing the convergence rate using the ratio $\twonorm{ \mathbf{r}_{l+m} }/\twonorm{ \mathbf{r}_l }$ is that it allows for a stage-by-stage analysis of the convergence process. This formulation captures the evolution of the convergence rate between any two iterations and helps interpret the typically nonlinear behavior of GMRES convergence.

\begin{remark}[Link with Cao~\cite{Cao1997NCB}]\label{rem:cao}
The convergence bound proposed in~\cite[Theorem 6]{Cao1997NCB} is given by:
\begin{equation*}
        \frac{\twonorm{ \mathbf{r}_{l+m} }}{\twonorm{ \mathbf{r}_l }}  
        \leq \kappa_{2}(\mathbf{V}) \max_{\lambda_i \neq \lambda_1} \abs{\frac{1-\frac{\lambda_i}{\lambda_1}}{1-\frac{\lambda_i}{\nu_1^{(l)}}}} \min_{q_{m}\in \mathcal{P}_{m}^1} \max_{\lambda_i \neq \lambda_1} \abs{ q_{m}(\lambda_i) },
    \end{equation*}
where $\lambda_1$ is the smallest eigenvalue of $\mathbf{A}$, and $\nu_1^{(l)}$ is an HR value at iteration $l$.
This bound focuses on the approximation of a single eigenvalue and includes the global condition number $\kappa_2(\mathbf{V})$, which reflects the influence of all eigenvectors and the overall non-normality of $\mathbf{A}$.

Note that the bound in \Cref{thm:harmonicRitzValues} also holds with $\sum_{\lambda_i\in \Lambda_J^c} \kappa(\lambda_i)$ replaced by $\kappa_2(\mathbf{V})$.
An advantage of $\sum_{\lambda_i\in \Lambda_J^c} \kappa(\lambda_i)$ is that as the iterations progress, more and more eigenvalues are well approximated by the HR values, so their condition numbers are removed from the sum.
Thus, the non-normality measured by the bound decreases.
We can expect that, asymptotically, this version of the bound is more accurate.
\end{remark}

The bound in \Cref{thm:harmonicRitzValues} provides a residual estimate based on eigenvalue separation using the eigendecomposition of $\mathbf{A}$. However, when $\mathbf{A}$ is non-normal or nearly defective, its eigenvalues alone may not fully capture the convergence behavior of GMRES.\ In such cases, a more robust analysis can be achieved using the concept of the \emph{pseudospectrum}, which accounts for the sensitivity of the spectrum to perturbations and highlights the influence of non-normality.

\begin{theorem}[Second residual estimate]\label{thm:harmonicRitzValues2}
    Let the assumptions of \Cref{lem:convergenceBound} hold and let $\Gamma_J^c$ be a finite union of disjoint, positively oriented Jordan curves with disjoint interiors, enclosing the eigenvalues in $\Lambda_J^c$ of $\mathbf{A}$ and no other eigenvalues.
    Suppose there exists a constant $\varepsilon_J^c > 0$ such that
    \[
        \left\|\left(  \mathbf{A} - z\mathbf{I} \right)^{-1} \right\|_2 \leq \frac{1}{\varepsilon_J^c}, \quad \forall z \in \Gamma_J^c.
    \]
    We assume that the HR values in $N_J^{(l)}$ lie outside $\Gamma_J^c$ and the regions it encloses.
    Then, for iteration indices $l$ and $l+m$, the residuals satisfy:
    \[
        \frac{\left\| \mathbf{r}_{l+m} \right\|_2}{\left\| \mathbf{r}_l \right\|_2} \leq \frac{\mathcal{L}(\Gamma_J^c)}{2\pi \varepsilon_J^c} \cdot \max_{z \in \Gamma_J^c} |s_J^l(z)| \cdot \min_{q_m \in \mathcal{P}_m^1} \max_{z \in \Gamma_J^c} |q_m(z)|,
    \]
    where $\mathcal{L}(\Gamma_J^c)$ denotes the total length of $\Gamma_J^c$.
\end{theorem}

\begin{proof}
    The spectral projector associated with $\Lambda_J^c$ (see~\cite[Sections~3.1.3--3.1.4]{Saad2011NML}) is defined by
    \begin{equation*}
        \mathbf{P}_{J}^{c} := \frac{-1}{2\iu\pi} \int_{\Gamma_{J}^c} {\left(\mathbf{A}-z\mathbf{I}\right)}^{-1}\mathrm{d}z.
    \end{equation*}
    Let $\mathbf{v}_i$ denote the $i$-th column of $\mathbf{V}$, so that $\mathbf{A} \mathbf{v}_i = \lambda_i \mathbf{v}_i$ and $\mathbf{V}^{-1} \mathbf{v}_i = \mathbf{e}_i$, the $i$-th canonical vector of $\C^N$.
    Then, for any polynomial $q_m$ and all $z \notin \sigma(\mathbf{A})$,
    \begin{equation}\label{eq:eigvec}
        q_m(\mathbf{A}) \mathbf{v}_i = q_m(\lambda_i) \mathbf{v}_i,
        \qquad
        {(\mathbf{A} - z\mathbf{I})}^{-1} \mathbf{v}_i = \frac{\mathbf{v}_i}{\lambda_i - z},
        \qquad
        \mathbf{S}_J^l \mathbf{v}_i = s_i \mathbf{v}_i,
    \end{equation}
    where the last identity follows from $\mathbf{S}_J^l \mathbf{v}_i = \mathbf{V} \, \mathrm{diag}(s_1, \ldots, s_N) \, \mathbf{e}_i$ by~\eqref{eq:SJl}, with $s_i = 0$ when $\lambda_i \in \Lambda_J$.
    Moreover, by the residue theorem, since $\Gamma_J^c$ is positively oriented and encloses the eigenvalues in $\Lambda_J^c$ and no other eigenvalues (see also~\cite[Theorem 3.3, Lemma 3.1]{Saad2011NML}),
    \begin{equation}\label{eq:pjc}
        \mathbf{P}_J^c \mathbf{v}_i
        = \left( \frac{1}{2\iu\pi} \int_{\Gamma_J^c} \frac{\mathrm{d}z}{z - \lambda_i} \right) \mathbf{v}_i
        = \begin{cases}
            \mathbf{v}_i & \text{if } \lambda_i \in \Lambda_J^c, \\
            \mathbf{0} & \text{if } \lambda_i \in \Lambda_J.
        \end{cases}
    \end{equation}

    Consequently, $\mathbf{P}_J^c \mathbf{S}_J^l \mathbf{v}_i = s_i \, \mathbf{P}_J^c \mathbf{v}_i = s_i \mathbf{v}_i = \mathbf{S}_J^l \mathbf{v}_i$ for $i = 1, \ldots, N$, and since the $\mathbf{v}_i$ form a basis of $\C^N$, $\mathbf{P}_J^c \mathbf{S}_J^l = \mathbf{S}_J^l$.
    Besides, $\mathbf{A}$ commutes with $\mathbf{P}_J^c$ (see~\cite[Section~3.1.4]{Saad2011NML}).
    Then, from \Cref{lem:convergenceBound}, $\mathbf{S}_J^l = \mathbf{P}_J^c \mathbf{S}_J^l$ and $q_m(\mathbf{A}) \mathbf{P}_J^c = \mathbf{P}_J^c \, q_m(\mathbf{A})$, we have:
    \begin{equation}\label{eq:boundPqS}
        \frac{\left\| \mathbf{r}_{l+m} \right\|_2}{\left\| \mathbf{r}_l \right\|_2} \leq \min_{q_m \in \mathcal{P}_m^1} \left\| \mathbf{P}_J^c \, q_m(\mathbf{A}) \, \mathbf{S}_J^l \right\|_2.
    \end{equation}

    We now show that
    \begin{equation}\label{eq:PqS}
        \mathbf{P}_J^c \, q_m(\mathbf{A}) \, \mathbf{S}_J^l = \frac{-1}{2\iu\pi} \int_{\Gamma_J^c} (\mathbf{A} - z\mathbf{I})^{-1} q_m(z) s_J^l(z) \, \mathrm{d}z.
    \end{equation}%
    As above, it is enough to check that both sides of~\eqref{eq:PqS} coincide on each $\mathbf{v}_i$.
    For the left-hand side, \eqref{eq:eigvec} and~\eqref{eq:pjc} give
    \[
        \mathbf{P}_J^c \, q_m(\mathbf{A}) \, \mathbf{S}_J^l \, \mathbf{v}_i = q_m(\lambda_i) \, s_i \, \mathbf{P}_J^c \mathbf{v}_i = q_m(\lambda_i) \, s_i \, \mathbf{v}_i,
    \]
    since $s_i = 0$ when $\lambda_i \in \Lambda_J$.
    For the right-hand side, by~\eqref{eq:eigvec}, it remains to show that
    \[
        \frac{-1}{2\iu\pi} \int_{\Gamma_J^c} \frac{q_m(z) \, s_J^l(z)}{\lambda_i - z} \, \mathrm{d}z
        = q_m(\lambda_i) \, s_i
        = \begin{cases}
            q_m(\lambda_i) \, s_J^l(\lambda_i) & \text{if } \lambda_i \in \Lambda_J^c, \\
            0 & \text{if } \lambda_i \in \Lambda_J.
        \end{cases}
    \]
    By assumption, the poles of $q_m \, s_J^l$, which are HR values in $N_J^{(l)}$, lie outside $\Gamma_J^c$, so that $q_m \, s_J^l$ is holomorphic inside and on $\Gamma_J^c$.
    The first case then follows from Cauchy's integral formula, since $\lambda_i$ lies inside $\Gamma_J^c$.
    In the second case, $\lambda_i$ lies outside $\Gamma_J^c$, so that $z \mapsto q_m(z) \, s_J^l(z)/(z - \lambda_i)$ is holomorphic inside and on $\Gamma_J^c$, so the integral vanishes by Cauchy's theorem.
    Note that $s_J^l$ is never evaluated at the eigenvalues in $\Lambda_J$, so that this holds even if some of them are HR values in $N_J^{(l)}$.

    Using~\eqref{eq:PqS} in~\eqref{eq:boundPqS}, we derive:
    \[
        \frac{\left\| \mathbf{r}_{l+m} \right\|_2}{\left\| \mathbf{r}_l \right\|_2} \leq \min_{q_m \in \mathcal{P}_m^1} \frac{1}{2\pi} \int_{\Gamma_J^c} |q_m(z)| \cdot |s_J^l(z)| \cdot \left\| {(\mathbf{A} - z\mathbf{I})}^{-1} \right\|_2 \, \mathrm{d}z.
    \]
    Bounding the integrand by its maximum on $\Gamma_J^c$, we find:
    \[
        \frac{\left\| \mathbf{r}_{l+m} \right\|_2}{\left\| \mathbf{r}_l \right\|_2}
        \leq \frac{\mathcal{L}(\Gamma_J^c)}{2\pi} \cdot \max_{z \in \Gamma_J^c} |s_J^l(z)| \cdot \min_{q_m \in \mathcal{P}_m^1} \max_{z \in \Gamma_J^c} |q_m(z)| \cdot \max_{z \in \Gamma_J^c} \left\| {(\mathbf{A} - z\mathbf{I})}^{-1} \right\|_2.
    \]
    Applying the resolvent bound yields the desired result:
    \[
        \frac{\left\| \mathbf{r}_{l+m} \right\|_2}{\left\| \mathbf{r}_l \right\|_2} \leq \frac{\mathcal{L}(\Gamma_J^c)}{2\pi \varepsilon_J^c} \cdot \max_{z \in \Gamma_J^c} |s_J^l(z)| \cdot \min_{q_m \in \mathcal{P}_m^1} \max_{z \in \Gamma_J^c} |q_m(z)|.
    \]
\end{proof}

\paragraph{Interpretation}
In \Cref{thm:harmonicRitzValues2}, the convergence bound involves a maximum over the contour $\Gamma_J^c$, which is a union of closed curves in the complex plane. The choice of this contour is restricted by the requirement that the resolvent norm $\|{(\mathbf{A} - z\mathbf{I})}^{-1}\|_2$ remains bounded above by $1/\eps_J^c$ for all $z \in \Gamma_J^c$.
Different choices of $\eps_J^c$ lead to different bounds.
In cases where $\mathbf{A}$ is highly non-normal, especially when eigenvalues are clustered around a non-normal eigenvalue, ensuring that this bound is useful in practice may force $\eps_J^c$ to be very small.
As a result, even if an HR value is close to such an eigenvalue, its influence can still be felt implicitly in the convergence bound via the smallness of $\eps_J^c$.

Then, both the \emph{pseudospectrum-based} bound (\Cref{thm:harmonicRitzValues2}) and the \emph{spectrum-based} bound (\Cref{thm:harmonicRitzValues}) are influenced by non-normal eigenvalues in $\boldsymbol{\Lambda_J}$ indirectly through the geometry (i.e.\ angles between eigenvectors) and norm behavior on the contour.
This influence decreases as increasingly more eigenvalues are approximated by the HR values. 

%% file: 3_helmholtz.tex
\section{Helmholtz problem and acceleration techniques}\label{sec:helmholtz}

In this section, we introduce the Helmholtz boundary value problem, the linear system arising from its finite element discretization, and techniques to accelerate the iterative solution process. We also discuss the difficulties caused by resonances and quasi-resonances occurring near the wavenumber, which can significantly impair convergence.

\subsection{Boundary value problem}\label{subsec:HelmholtzProblem}

Let $\Omega \subset \mathbb{R}^d$ be a bounded domain with Lipschitz boundary $\Gamma$, which is partitioned into three disjoint parts: $\Gamma = \Gamma_\mathrm{D} \cup \Gamma_\mathrm{N} \cup \Gamma_\mathrm{R}$, corresponding to Dirichlet, Neumann, and Robin boundary segments, respectively.
The Helmholtz problem is defined as follows:
\begin{align}\label{eq:HelmholtzPb}
    \left\{
    \begin{aligned}
        -\Delta u - k^2 u                     &= f   && \text{in } \Omega,     \\
        u                                     &= g_\mathrm{D} && \text{on } \Gamma_\mathrm{D}, \\
        \partial_{\boldsymbol{n}} u           &= g_\mathrm{N} && \text{on } \Gamma_\mathrm{N}, \\
        \partial_{\boldsymbol{n}} u + \iu k u &= 0   && \text{on } \Gamma_\mathrm{R}, \\
    \end{aligned}
    \right.
\end{align}
with source term $f$ and boundary data $g_\mathrm{D}$ and $g_\mathrm{N}$.

A variational formulation is obtained by multiplying the Helmholtz equation by a test function $v \in V_0 := \{v \in H^1(\Omega) \;|\; v = 0 \text{ on } \Gamma_\mathrm{D}\}$ and integrating by parts.
The solution space is $V := \{v \in H^1(\Omega) \;|\; v = g_\mathrm{D} \text{ on } \Gamma_\mathrm{D}\}$.
Taking into account the boundary conditions yields the weak form:
\begin{equation*}
    \left|\ 
    \begin{aligned}
        \text{Find \( u \in V\) such that $a(u, v) = \ell(v)$ for all $v \in V_0$.}
    \end{aligned}
    \right.
\end{equation*}
Here, the sesquilinear form $a$ and linear form $\ell$ are defined by
\begin{equation*}
    a(u,v) = \int_{\Omega} \left( \nabla u \cdot \nabla \bar{v} - k^2 u \bar{v} \right) \mathrm{d}\Omega + \iu k \int_{\Gamma_\mathrm{R}} u \bar{v} \, \mathrm{d}\Gamma,
    \quad
    \ell(v) = \int_{\Omega} f \bar{v} \, \mathrm{d}\Omega + \int_{\Gamma_\mathrm{N}} g_\mathrm{N} \bar{v} \, \mathrm{d}\Gamma.
\end{equation*}

We also consider problems defined in unbounded domains.
In that case, the Helmholtz equation is supplemented with the Sommerfeld radiation condition at infinity~\cite{Sommerfeld1912DGF,Schot1992EYS}:
\begin{equation*}
     \lim_{r \rightarrow \infty} r \left( \frac{\partial u}{\partial r} - \iu ku\right) = 0, \quad r = \sqrt{x^2+y^2},
\end{equation*}
and the solutions are sought in $H^1_{\text{loc}}(\Omega)$, the space of functions locally in $H^1(\Omega)$.
In practice, numerical simulations are carried out on truncated computational domains equipped with artificial boundary conditions—such as absorbing boundary conditions~\cite{EngquistMajda1977ABC,ModaveGeuzaineEtAl2020CTH} or perfectly matched layers (PMLs)~\cite{BermudezHervellaNietoEtAl2004EBP, BermudezHervellaNietoEtAl2007OPM, BeriotModave2020APM, BecacheDhiaEtAl2006PML,Berenger1994PML, MarchnerBeriotEtAl2021SPM, TurkelYefet1998APB}.
In this work, we employ PMLs (see \Cref{sec:scatteringProblems}).

The Helmholtz problems may be ill-posed for certain wavenumbers $k$.
For the bounded case (Problem~\eqref{eq:HelmholtzPb}), these values, called \emph{resonances}, appear if $\Gamma_{\mathrm{R}} = \emptyset$. They are real and correspond to the eigenvalues of the associated Laplace problem on $\Omega$.
Similar phenomena occur in unbounded cases involving configurations such as open cavities or inhomogeneous media~\cite{Moitier2019EME}.
These resonances arise from nonphysical complex values of $k$ (called \emph{scattering resonances}) with small negative imaginary parts.
These values imply the existence of real values of $k$, called \emph{quasi-resonances}, which are associated with localized quasi-modes, and for which the Helmholtz problems are well-posed~\cite{Stefanov2000RRA}. Quasi-modes are defined as follows:

\begin{definition}[{Quasi-modes~\cite[Definition 1.1]{MarchandGalkowskiEtAl2022AGH}}]\label{def:quasimodes}
    A family of quasi-resonances with associated quasi-modes of quality $\epsilon(k)$ is a sequence
    \[    {\left\{\left( k_j, u_j \right)\right\}}_{j=1}^\infty \subset \R \times H^1_{\text{loc}}(\Omega),\]
    such that $k_j \to \infty$ when $j \to \infty$, and there exists a compact set $K
        \subset \Omega$ such that, for all $i,j$,
    \[    \textnormal{supp}(u_i) \subset K, \quad \norm{-\Delta u_i - k_i^2 u_i}{L^2(\Omega)} \leq \epsilon(k_i), \quad \text{and} \quad \abs{\left(u_i,u_j\right)-\delta_{i,j}} \leq \epsilon(k_i),\]
   where $\left(u,v\right) := \int_\Omega u \bar{v} \, \mathrm{d}\Omega$ denotes the $L^2(\Omega)$ inner product.
\end{definition}
Depending on the problem, the decay of \(\epsilon(k_j)\) when $j \rightarrow \infty$ can be polynomial (known as \emph{weak-trapping}) or exponential (\emph{strong-trapping})~\cite{ChandlerWildeSpenceEtAl2020HFB}.

As we shall see, when the wavenumber $k$ approaches a resonance or quasi-resonance, the linear system resulting from the discretization of the problem becomes significantly harder to solve.

\subsection{Discretization}\label{subsec:numericalDiscretization}
To numerically solve the Helmholtz problem introduced above, we apply a polynomial finite element discretization on a triangular conforming mesh of the domain.
The average element size is $h$. We use shape functions based on integrated Legendre polynomial~\cite{SzaboBabuska2021FEA,PetersenDreyerEtAl2006AFS}.
The discrete solution $u_h$ is sought in a finite-dimensional subspace $V_h \subset V$ and satisfies the variational formulation
\begin{align*}
    a(u_h, v_h) = \ell(v_h), \qquad \forall v_h \in V_{h,0} := V_h \cap V_0.
\end{align*}
The solution $u_h$ can be expanded as
\begin{align*}
    u_h(\boldsymbol{x}) = \sum_{j=1}^{N} u_j \phi_j(\boldsymbol{x}) + \sum_{j=N+1}^{N+N_\mathrm{D}} \hat{u}_{j-N} \phi_j(\boldsymbol{x}),
\end{align*}
 with ${\{\phi_j\}}_{j=1}^N$ forming a basis for $V_{h,0}$ and ${\{\phi_j\}}_{j=N+1}^{N+N_\mathrm{D}}$ are the $N_\mathrm{D}$ basis functions associated with the degrees of freedom lying on $\Gamma_\mathrm{D}$, which do not vanish on that boundary. 
The coefficients $\{u_j\}$ are unknowns, while $\{\hat{u}_j\}$ encode the Dirichlet data on $\Gamma_\mathrm{D}$.
Substituting this representation into the variational formulation and choosing the test functions from $V_{h,0}$ leads to the linear system
\begin{align}
    \label{eq:linearSystem2}
    \left|\
    \begin{aligned}
        \text{Find $\mathbf{u} \in \mathbb{C}^N$ such that } \mathbf{A}\mathbf{u} = \mathbf{b}, 
    \end{aligned}
    \right.
\end{align}
with system matrix and right-hand side defined by
\[
    A_{i,j} = a(\phi_j, \phi_i), \quad b_i = \ell(\phi_i) - \sum_{j=N+1}^{N+N_\mathrm{D}} a(\phi_j, \phi_i)\, \hat{u}_{j-N}, \quad \text{for } i,j = 1,\dots,N.
\]

The properties of the discrete system mirror those of the continuous problem:
\begin{itemize}
    \setlength{\itemsep}{0pt}
    \setlength{\parskip}{0pt}
    \setlength{\parsep}{0pt}
    \item The matrix $\mathbf{A}$ is \emph{non-singular}, since $k$ is assumed not to correspond to a resonance.
    \item The sesquilinear form is typically not coercive. The matrix $\mathbf{A}$ is often highly \emph{indefinite}, and this indefiniteness increases with $k$. The numerical range $W(\mathbf{A}) = \left\{ \mathbf{x}^*\mathbf{A} \mathbf{x} \, : \, \left\|\mathbf{x}\right\|_2 = 1 \right\}$ contains zero.
    \item The matrix $\mathbf{A}$ may also be \emph{non-normal}, meaning its eigenvectors are not mutually orthogonal, for instance when PML are used (see \Cref{sec:scatteringProblems}).
\end{itemize}

These features make the Helmholtz linear system particularly challenging for iterative solvers.
At high frequencies or near (quasi-)resonances, these difficulties become more pronounced:~\cite{GalkowskiMarchandEtAl2021ETH} shows that quasi-modes give rise to small eigenvalues.
In the boundary element setting,~\cite{MarchandGalkowskiEtAl2022AGH} analyzes how GMRES convergence deteriorates near quasi-modes.
Here, with a finite element discretization, we likewise observe the emergence of small eigenvalues and corresponding GMRES stagnation.


\subsection{Acceleration techniques and preconditioning}\label{subsec:acceleration}

In this section, we describe two techniques to accelerate the convergence of GMRES for solving the discrete problem~\eqref{eq:linearSystem2}: the deflation technique and the CSL preconditioning technique.
We also assume that the Assumption\ \ref{hyp:Adiag} holds in this section.

\subsubsection*{Deflation}\label{subsec:deflation}

The principle of the deflation is to remove a part of the spectrum of $\mathbf{A}$  in order to speed up the GMRES convergence on the modified matrix.
This is done by solving a small problem in a subspace, called the \emph{deflation subspace}, which is designed to remove the part of $\mathbf{A}$ that slows down convergence.
This process relies on two bases, corresponding to the columns of matrices $\mathbf{Y}$ and $\mathbf{Z}$ below, which must be chosen with care.
The method was initially developed for symmetric problems~\cite{FrankVuik2001CDB,NabbenVuik2008CAV,TangNabbenEtAl2009CTL}, but the theoretical framework was later generalized to non-Hermitian problems~\cite{ErlanggaNabben2008DBP,GarciaRamosKehlEtAl2020PDM,SpillaneSzyld2024NCA}.
First, we provide a general overview of deflation. After that, we discuss the choice of $\mathbf{Y}$ and $\mathbf{Z}$. Then, in~\Cref{sec:cavityProblem,sec:scatteringProblems}, the deflation subspaces will be spanned by modes or quasi-modes of the physical problems.

\begin{definition}{{(Deflation matrices)}}\label{def:def}
    Let $\mathbf{Y}, \mathbf{Z} \in \C^{N \times n_{\mathrm{def}}}$ be two full rank matrices such that $\ker(\mathbf{Y}^{*}) \cap \range{\mathbf{AZ}} = \{0\}$, and let
    \begin{equation*}
        \mathbf{E} := \mathbf{Y}^{*}\mathbf{A}\mathbf{Z}\in \C^{n_{\mathrm{def}}\times n_{\mathrm{def}}}
        \quad\text{and}\quad
        \mathbf{Q} := \mathbf{Z} \mathbf{E}^{-1} \mathbf{Y}^{*}.
    \end{equation*}
    The \emph{deflation matrices} are defined as
    \begin{equation*}
        \mathbf{P}_\mathrm{def} := \mathbf{I} - \mathbf{A} \mathbf{Q}
        \quad\text{and}\quad
        \mathbf{Q}_\mathrm{def} := \mathbf{I} - \mathbf{Q} \mathbf{A}.
    \end{equation*}
\end{definition}
Matrix $\mathbf{E}$ is the restriction of $\mathbf{A}$ to the deflation subspace. It is non-singular and a priori easy to invert, since $n_{def}$ is assumed to be small.

The deflation matrices are projectors onto the complement of the deflation subspace and satisfy
$\mathbf{P}_\mathrm{def}^{2} = \mathbf{P}_\mathrm{def}$,
$\mathbf{Q}_\mathrm{def}^{2} = \mathbf{Q}_\mathrm{def}$ and
$\mathbf{P}_\mathrm{def} \mathbf{A} = \mathbf{A} \mathbf{Q}_\mathrm{def}$, see\ \cite[Proposition 3.1]{GarciaRamosKehlEtAl2020PDM}.
Applying $\mathbf{P}_\mathrm{def}$ or $\mathbf{Q}_\mathrm{def}$ to $\mathbf{A}$ removes the components associated with the deflation subspace. The deflated problem is~\cite[Lemma 3.2]{GarciaRamosKehlEtAl2020PDM}
\begin{align}
    \label{eq:deflatedLinearSystem}
    \left|\ 
    \begin{aligned}
      \text{Find $\tilde{\mathbf{u}} \in \mathbb{C}^{N}$ such that $(\mathbf{P}_\mathrm{def} \mathbf{A}) \tilde{\mathbf{u}} = (\mathbf{P}_\mathrm{def} \mathbf{b})$},
    \end{aligned}
    \right.
\end{align}
and the solution of the original problem~\eqref{eq:linearSystem2} is obtained by using
\begin{equation*}
    \mathbf{u} = \mathbf{Q} \mathbf{b} + \mathbf{Q}_\mathrm{def} \tilde{\mathbf{u}}.
\end{equation*}
This can be shown by noting that $\mathbf{u}$ can be decomposed as
\begin{equation}
    \label{eq:decompositionDeflation}
    \mathbf{u} = \left( \mathbf{I} - \mathbf{Q}_\mathrm{def} \right)\mathbf{u} + \mathbf{Q}_\mathrm{def} \mathbf{u} =\mathbf{Q} \mathbf{b} + \mathbf{Q}_\mathrm{def} \mathbf{u}.
\end{equation}
Using the property $\mathbf{P}_\mathrm{def} \mathbf{A} = \mathbf{A} \mathbf{Q}_\mathrm{def}$,~\eqref{eq:deflatedLinearSystem} gives
\begin{equation*}
    \mathbf{A} \mathbf{Q}_\mathrm{def} \tilde{\mathbf{u}} = \mathbf{P}_\mathrm{def} \mathbf{A} \tilde{\mathbf{u}} = \mathbf{P}_\mathrm{def} \mathbf{b} = \mathbf{P}_\mathrm{def} \mathbf{A} \mathbf{u} = \mathbf{A} \mathbf{Q}_\mathrm{def} \mathbf{u}.
\end{equation*}
Since $\mathbf{A}$ is assumed to be non-singular by Assumption\ \ref{hyp:Adiag}, one has $\mathbf{Q}_\mathrm{def} \tilde{\mathbf{u}} = \mathbf{Q}_\mathrm{def} \mathbf{u}$.

Because $\mathbf{P}_\mathrm{def}$ projects out the deflation subspace, the matrix $\mathbf{P}_\mathrm{def} \mathbf{A}$ is singular, with zero as an eigenvalue of multiplicity $n_{\mathrm{def}}$.
But GMRES applied to this system still converges to a solution if the following condition is satisfied~\cite[Theorem 3.4]{GarciaRamosKehlEtAl2020PDM}:
\begin{equation}\label{eq:singularSystem}
    \ker(\mathbf{P}_\mathrm{def} \mathbf{A}) \cap \range{\mathbf{P}_\mathrm{def} \mathbf{A}} = \{0\}.
\end{equation}

The choice of $\mathbf{Y}$ and $\mathbf{Z}$ is not trivial; several options have been explored in the literature, such as exact eigenvectors, vectors derived from the multigrid method or domain decomposition, see e.g.~\cite{VuikSegalEtAl2002CVD,ErlanggaNabben2008DBP,MacLachlanOosterlee2008AMS,TangNabbenEtAl2009CTL,GarciaRamosKehlEtAl2020PDM}.
In the numerical experiments we consider in~\Cref{sec:cavityProblem,sec:scatteringProblems}, matrix $\mathbf{A}$ has small eigenvalues.
The associated eigenvectors can be approximated by modes or quasi-modes of the physical problems, and can be used empirically to deflate the small eigenvalues.
Thus, in what follows, the columns of $\mathbf{Z}$ will be the projections of these modes or approximated quasi-modes onto the finite-element space.
Based on this physical intuition, we take $\mathbf{Y} = \mathbf{Z}$.

If exact eigenvectors are chosen, then $\mathbf{E}$ is non-singular (diagonal when $\mathbf{A}$ is normal), since the problem is well-posed.
In this case, the hypothesis of \Cref{def:def} holds and the condition~\eqref{eq:singularSystem} is automatically satisfied, so GMRES is guaranteed to converge to the solution.
A similar situation can be expected when using approximate eigenvectors.

Although the deflated system resembles a system preconditioned on the left by $\mathbf{P}_\mathrm{def}$, GMRES applied to~\eqref{eq:deflatedLinearSystem} still minimizes the residual norm of the original problem.
This follows from the decomposition~\eqref{eq:decompositionDeflation} and the definition of $\mathbf{P}_\mathrm{def}$; see~\cite[Theorem 3.1]{ErlanggaNabben2008DBP}.

Each GMRES iteration becomes more expensive due to the application of $\mathbf{P}_\mathrm{def}$.
Nevertheless, if $n_{\mathrm{def}}$ is small, computing $\mathbf{E}^{-1}$ is inexpensive.
In addition, if $\mathbf{Z}$ is a full matrix (as is the case if its columns are eigenvectors), noting that $\mathbf{P}_\mathrm{def} \mathbf{A} = \mathbf{A} - (\mathbf{A} \mathbf{Z}) (\mathbf{E}^{-1} \mathbf{Z}^{*} \mathbf{A})$, the $N\times n_{\mathrm{def}}$ matrix $\mathbf{A} \mathbf{Z}$ and the $n_{\mathrm{def}} \times N$ matrix $\mathbf{E}^{-1} \mathbf{Z}^{*} \mathbf{A}$ can be stored.
Then, applying $\mathbf{P}_\mathrm{def} \mathbf{A}$ requires only $4N n_{\mathrm{def}} - N - n_{\mathrm{def}}$ additional operations compared to applying $\mathbf{A}$.
In practice, $\mathbf{Z}$ can be sparse (as in~\Cref{sec:scatteringProblems}), and this reduces the computational cost.

\subsubsection*{Deflation and preconditioning}\label{subsec:deflationPreconditioning}

Having described deflation as a technique to accelerate GMRES by removing selected eigenvalues, we now turn to another complementary approach: preconditioning.
That approach aims to improve the spectral properties of the system by pre- or post-multiplying the matrix system by an approximation of its inverse.
We then investigate how deflation can be combined with preconditioning.

We consider the CSL preconditioner as a representative example, following~\cite{ErlanggaOosterleeEtAl2006NMB,ErlanggaVuikEtAl2006CMI,CoolsVanroose2013LFA,GanderGrahamEtAl2015AGH,RamosNabben2020TLS,RamosSeteEtAl2021PHE}.
The preconditioner is based on solving a modified Helmholtz problem with a complex-shifted wavenumber.
\begin{equation*}
    \int_{\Omega} \left( \nabla u \cdot \nabla \bar{v} - (k^2 + \iu \varepsilon) u \bar{v} \right) \, \mathrm{d}\Omega + \iu k\int_{\Gamma_\mathrm{R}} u\bar{v} \, \mathrm{d}\Gamma = \int_{\Omega} f \bar{v} \, \mathrm{d}\Omega + \int_{\Gamma_\mathrm{N}} g_\mathrm{N} \bar{v} \, \mathrm{d}\Gamma, \quad \forall v \in V, \ \varepsilon > 0.
\end{equation*}
As suggested in~\cite{GanderGrahamEtAl2015AGH}, we use $\varepsilon = k$. Discretizing this problem yields the matrix $\mathbf{A}_\varepsilon$, and the original system is then right-preconditioned as:
\begin{align}
    \label{eq:rightPreconditionedSystem}
    \left|\ 
    \begin{aligned}
        \text{Find $\mathbf{x} \in \mathbb{C}^{N}$ such that } (\mathbf{A} \mathbf{A}_\varepsilon^{-1}) \mathbf{x} = \mathbf{b},
    \end{aligned}
    \right.
\end{align}
with the solution recovered as $\mathbf{u} = \mathbf{A}_\varepsilon^{-1} \mathbf{x}$.

The CSL preconditioner aims to reduce iteration counts without introducing excessive computational cost.
Its assembly is inexpensive, since the mass and stiffness matrices are reused from the original Helmholtz discretization.
However, applying $\mathbf{A}_\varepsilon^{-1}$ can be as costly as solving the original system.
Multigrid methods are often employed to efficiently deal with this challenge, but they can be difficult to implement and tune effectively.
Here, we rather use an incomplete LU (ILU) factorization as an efficient approximation.
Each GMRES iteration is then more expensive due to the cost of applying this approximate inverse.


To combine the CSL preconditioning and deflation techniques, we substitute the preconditioned formulation~\eqref{eq:rightPreconditionedSystem} into the deflated system~\eqref{eq:deflatedLinearSystem}.
This leads to the doubly modified system:
\begin{align}
    \label{eq:preconditionedDeflatedLinearSystem}
    \left|\ 
    \begin{aligned}
        \text{Find $\tilde{\mathbf{x}} \in \mathbb{C}^{N}$ such that } (\mathbf{P}_\mathrm{def} \mathbf{A} \mathbf{A}_\varepsilon^{-1}) \tilde{\mathbf{x}} = \mathbf{P}_\mathrm{def} \mathbf{b}.
    \end{aligned}
    \right.
\end{align}
The final solution is reconstructed using
\[
    \mathbf{u} = \mathbf{Q} \mathbf{b} + \mathbf{Q}_\mathrm{def} \mathbf{A}_\varepsilon^{-1} \tilde{\mathbf{x}}.
\]
If the conditions
\begin{equation*}
    \ker(\mathbf{Z}^{*}) \cap \range{\mathbf{AZ}} = \{0\} \quad \text{and} \quad \ker(\mathbf{Z}^{*}) \cap \range{\mathbf{A}_\eps^{-1}\mathbf{Z}} = \{0\},
\end{equation*}
hold, then GMRES applied to this system converges for all initial guesses~\cite[Theorem 3.4]{GarciaRamosKehlEtAl2020PDM}.
In practice, $\mathbf{A}_\varepsilon^{-1}$ is approximated by its ILU factorization.
Each GMRES iteration then involves the additional cost of deflation mentioned in \Cref{subsec:deflation} plus the cost of applying the preconditioner.

Alternatively, the deflation and preconditioning ideas can be combined additively, yielding the so-called \emph{additive coarse grid correction} method~\cite{BastianWittumEtAl1998AMM}.
This leads to the modified system:
\begin{align*}
    \left|\ 
    \begin{aligned}
        \text{Find $\mathbf{x} \in \mathbb{C}^{N}$ such that } \mathbf{A} \left( \mathbf{A}_\varepsilon^{-1} + \mathbf{Q} \right) \mathbf{x} = \mathbf{b},
    \end{aligned}
    \right.
\end{align*}
with the solution given by $\mathbf{u} = (\mathbf{A}_\varepsilon^{-1} + \mathbf{Q}) \mathbf{x}$.
This approach has the same additional cost per iteration of~\eqref{eq:preconditionedDeflatedLinearSystem}, but it does not require the initial computation of $\mathbf{A} \mathbf{Z}$ and $\mathbf{Z}^{*} \mathbf{A}$.
However, since the numerical results obtained with this method are nearly identical to those with deflation, and because deflation enjoys a more developed theoretical foundation, we do not pursue this approach further in this work.

%% file: 4_experiment_cavity.tex
\section{Numerical experiments: a cavity problem with resonances}\label{sec:cavityProblem}

In this section and the next, we present numerical experiments designed to illustrate the specific challenges of solving Helmholtz problems.
These benchmarks serve to validate the interpretation tools and acceleration techniques discussed previously.
The first set of experiments focuses on a simple cavity problem with a wavenumber chosen close to resonance, a setting known to lead to slow convergence.


\subsection*{Description of the benchmark}
We solve the Helmholtz equation on the unit square domain $\Omega = {(0,1)}^2$, subject to homogeneous Dirichlet boundary conditions on $\partial \Omega$ and a constant source term equal to one in the interior.
The problem admits a unique solution provided that the square of the wavenumber does not coincide with an eigenvalue of the Laplacian operator.

For this configuration, the Laplacian eigenvalues and eigenfunctions are given by
\[
    k_{n,m}^2 = \pi^2(n^2 + m^2), \quad u_{n,m}(x,y) = \sin(n\pi x)\sin(m\pi y), \quad n,m \in \N.
\]
After discretization, the system matrix $\mathbf{A}$ is real, symmetric, and therefore normal.
Its eigenvalues are real, and it admits an orthonormal basis of eigenvectors.

All experiments are conducted using a Matlab implementation.
We use $P_2$ continuous finite elements on a uniform mesh with characteristic size $h = 1/32$, leading to approximately $N_\mathrm{dof} = 5{,}000$ degrees of freedom (see~\Cref{fig:cavityPlot}).
The mesh resolution is chosen so that the number of degrees of freedom per wavelength satisfies $D_\lambda = (\sqrt{N_\mathrm{dof}} - 1)\,\lambda \geq 30$, where $\lambda = 2\pi/k  $ denotes the wavelength.

\begin{figure}[tb]
    \centering
    \begin{subfigure}{0.24\textwidth}
        \subcaption{Mesh}
        \includegraphics[height=0.15\textheight]{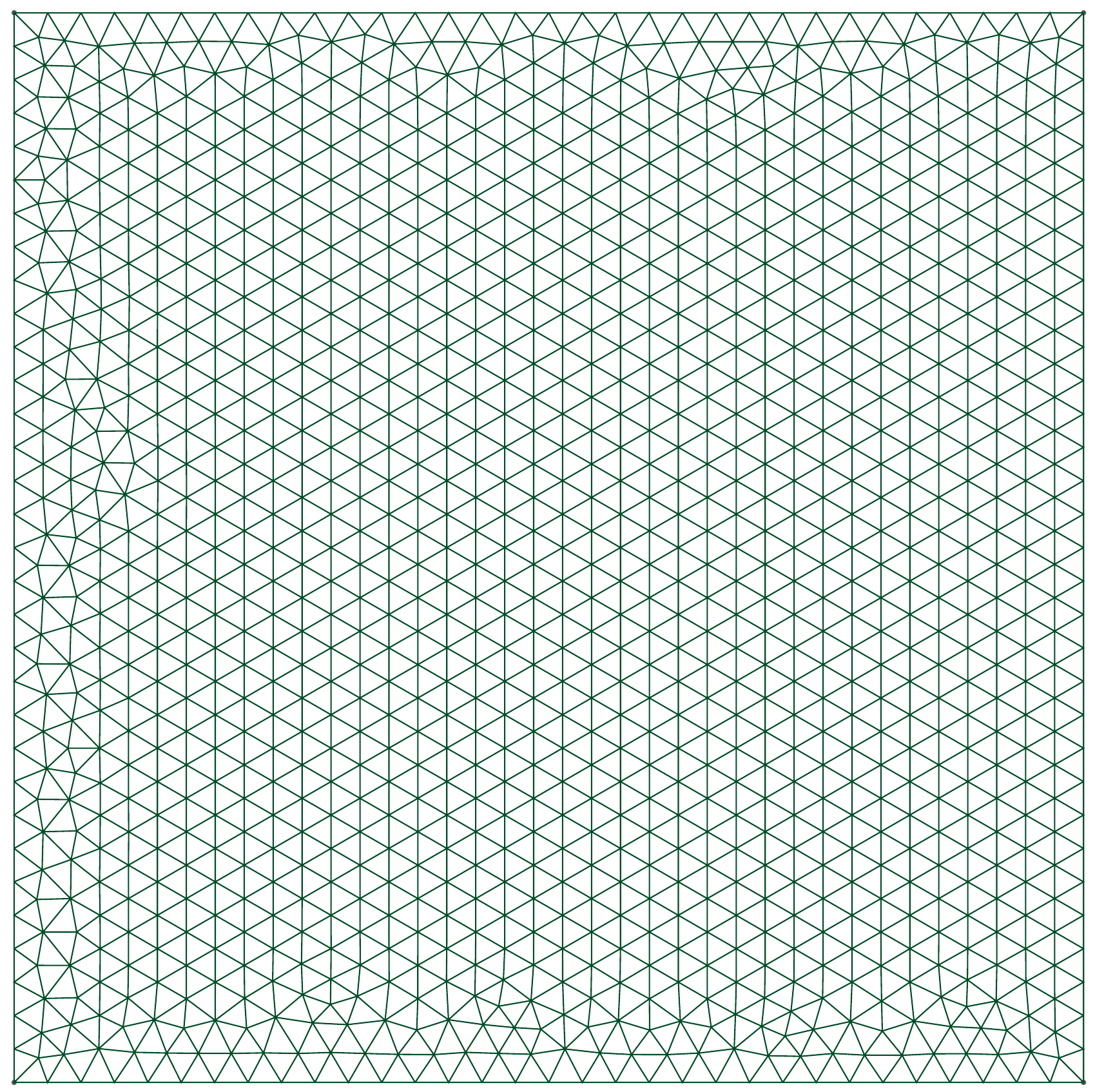}
    \end{subfigure}
    \begin{subfigure}{0.24\textwidth}\label{fig:mode3.3}
        \subcaption{Mode (3,3)}
        \includegraphics[height=0.15\textheight]{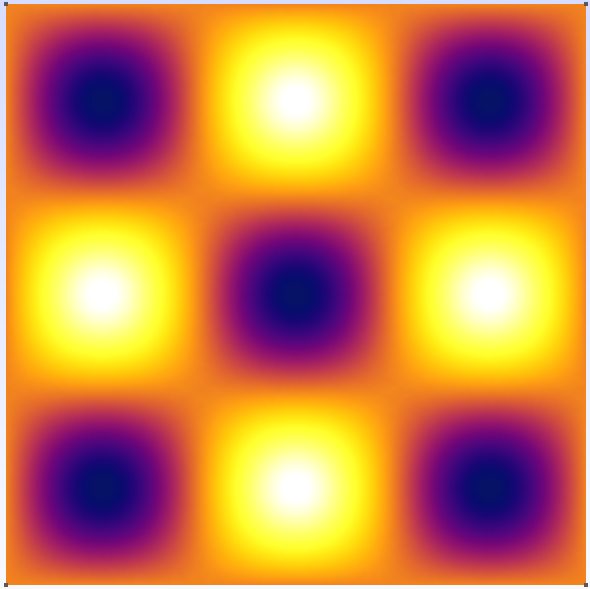}
    \end{subfigure}
    \begin{subfigure}{0.24\textwidth}
        \subcaption{Mode (2,1)}
        \includegraphics[height=0.15\textheight]{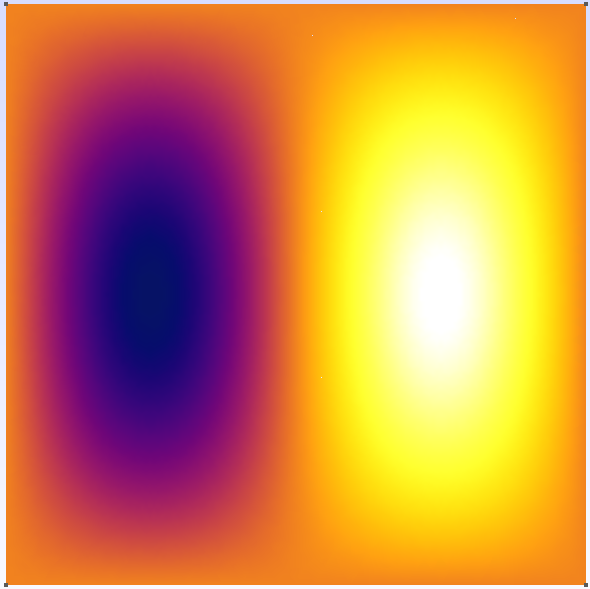}
    \end{subfigure}
    \begin{subfigure}{0.24\textwidth}
        \subcaption{Mode (1,4)}
        \includegraphics[height=0.15\textheight]{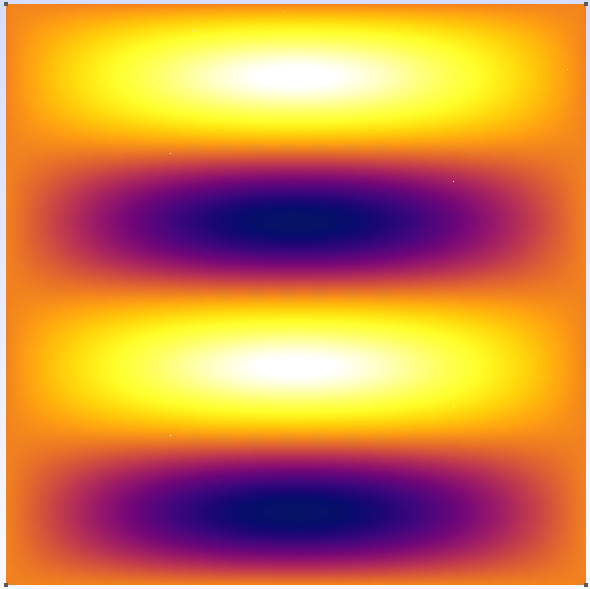}
    \end{subfigure}
    \caption{Cavity benchmark. Mesh for $h=1/32$, and three modes of the Laplacian used for deflation.}\label{fig:cavityPlot}
\end{figure}


\subsection*{Influence of the wavenumber on the relative error and the number of iterations}

\begin{figure}[!t]
    \centering
    \begin{subfigure}{.5\textwidth}
        \centering
        \caption{Relative $L^2$-error}
        \includegraphics{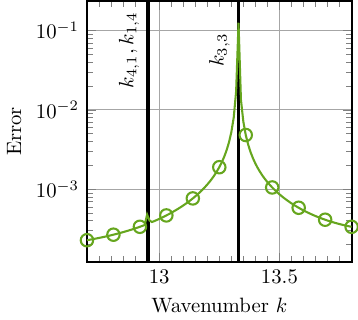}\label{fig:cavityError}
    \end{subfigure}\hfill
    \begin{subfigure}{.5\textwidth}
        \centering
        \caption{Number of GMRES iterations}
        \includegraphics{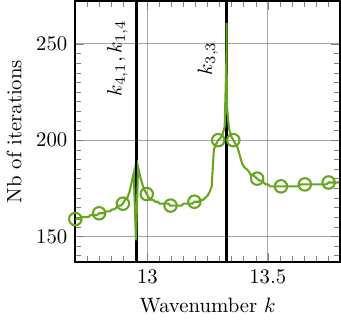}\label{fig:cavityNbit}
    \end{subfigure}
    \caption{Cavity benchmark. (a) Relative $L^2$-error between the iterative solution and the finite element projection of the analytical solution. (b) Number of GMRES iterations to reach the tolerance $\mathsf{tol} = 10^{-6}$, without acceleration technique. The vertical lines correspond to resonance wavenumbers. There are 122 sample points distributed uniformly, and 3 points located around the peak $k_{3,3}$.}\label{fig:cavityErrorNbit}
\end{figure}

We now analyze the behavior of the solution and solver convergence as a function of the wavenumber $k$ for the cavity benchmark introduced above.

In \Cref{fig:cavityErrorNbit}, we plot the relative error and the number of GMRES iterations required to reduce the relative residual below $10^{-6}$, as functions of the wavenumber $k$.
The selected range of $k$ includes two resonance values: \(k_{3,3} = 3\sqrt{2}\pi \approx 13.33\) and \(k_{4,1} = k_{1,4} \approx 12.95\), corresponding to eigenmodes of the Laplacian.

A pronounced peak in the relative error appears near \(k_{3,3}\), while no such peak is observed near \(k_{4,1}\) (see \Cref{fig:cavityError}).
This difference is explained by the structure of the source term:  it excites only those modes for which both indices \(n\) and \(m\) are odd. 
As a result, modes such as \(u_{4,1}\) and \(u_{1,4}\), which are not excited, contribute negligibly to the solution, and no significant error is observed even when $k$ approaches their corresponding eigenvalues.
When $k$ coincides with the resonance frequency, the problem becomes singular and numerical resolution is meaningless.

By contrast, the number of GMRES iterations is sensitive to the proximity of \emph{any} eigenmode, regardless of whether it is excited by the source term.
This is evident in \Cref{fig:cavityNbit}, where two iteration peaks appear near both resonance values.
This suggests that the convergence of GMRES deteriorates when the wavenumber is close to an eigenvalue, a behavior we explore in more detail in the following discussion.


\subsection*{Influence of the wavenumber on the residual history}

\begin{figure}[!bt]
    \centering
    \centering
    \includegraphics[width=0.79\textwidth]{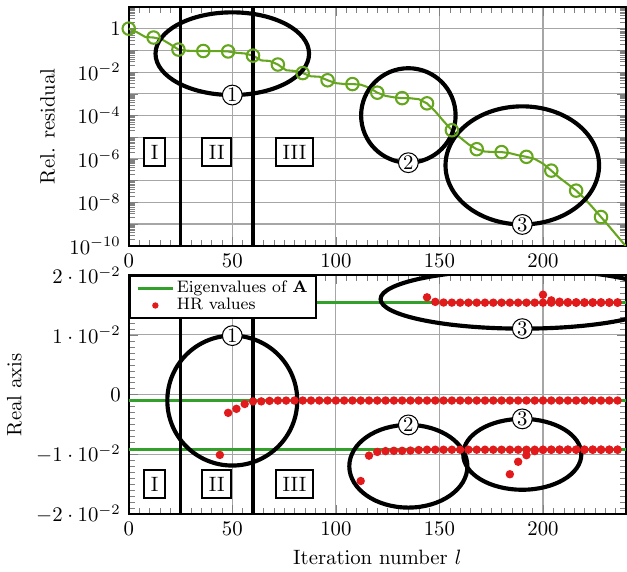}
    \caption{Cavity benchmark.
        Convergence analysis for $k=\hat{k}=3.01\sqrt{2}\pi$ near the resonance wavenumber $k_{3,3}$, without acceleration technique.
        The history of the relative residual over the iterations is shown at the top, and the trajectory of the HR values with the smallest magnitudes is shown at the bottom.
        The horizontal green lines correspond to the eigenvalues of $\mathbf{A}$ with the smallest magnitudes.}\label{fig:cavityTrajRes}
\end{figure}

To better understand the convergence behavior near resonance, we now examine the residual history of GMRES without acceleration for the wavenumber $\hat{k} = 3.01\sqrt{2}\pi \approx 13.37$, which lies close to the resonance $k_{3,3}$.
At this frequency, one eigenvalue of the system matrix $\mathbf{A}$ lies close to the origin, and convergence issues are expected. As shown in \Cref{fig:cavityTrajRes} (top), the residual exhibits several distinct plateau phases followed by sharp drops (highlighted by black circles labeled 1, 2, and 3).
This pattern can be interpreted using the HR values and the convergence bound in \Cref{thm:harmonicRitzValues}. 

\Cref{fig:cavityTrajRes} (bottom) tracks the evolution of the HR values near zero throughout the GMRES iterations.
These values, which are real in this case, converge toward the eigenvalues of $\mathbf{A}$ (indicated by green lines), as expected. The first plateau in the residual history aligns with the moment when an HR value begins to approach the smallest eigenvalue $\lambda_1$ (circle 1 in both figures).
To analyze this behavior, we apply \Cref{thm:harmonicRitzValues} with $\Lambda_J^c = \sigma(\mathbf{A}) \setminus \{\lambda_1\}$.
Note that since $\mathbf{A}$ is normal, the convergence bound of \Cref{thm:harmonicRitzValues} can be written with $\kappa_2(\mathbf{V})=1$ (see Remark\ \ref{rem:cao}).
Then, we examine the quantity
\begin{equation}
\label{eq:boundThm2.5}
\max_{\lambda_i \in \Lambda_J^c} \abs{s_J^{l}(\lambda_i)} 
= \max_{\lambda_i \in \Lambda_J^c} \abs{\frac{ \big(1 - \lambda_i/\lambda_1\big)}{ \big(1 - \lambda_i/\nu_1^{(l)}\big)}}
= \max_{\lambda_i \in \Lambda_J^c} \abs{\frac{\nu_1^{(l)}}{\lambda_1}} \abs{\frac{\lambda_1 - \lambda_i}{\nu_1^{(l)}- \lambda_i}}.
\end{equation}

We identify three distinct phases in the evolution of the GMRES residual, labeled $\mathrm{I}$, $\mathrm{II}$, and $\mathrm{III}$ in both plots of \Cref{fig:cavityTrajRes}.
Each phase reflects a different stage in the approximation of the smallest eigenvalue $\lambda_1$ by an HR value $\nu_1^{(l)}$, and corresponds to a different regime in the theoretical bound from \Cref{thm:harmonicRitzValues}.

\begin{itemize}
\setlength{\itemsep}{0pt}
\setlength{\parskip}{0pt}
\setlength{\parsep}{0pt}
\item \textbf{Phase I:\ HR values are far from $\lambda_1$}.  
At this stage, $\nu_1^{(l)}$ is not yet close to $\lambda_1$, and both terms in \Cref{eq:boundThm2.5} are large. The bound from \Cref{thm:harmonicRitzValues} is then loose, and cannot provide an accurate description of the convergence rate.

\item \textbf{Phase II:\ HR value approaches $\lambda_1$}.  
As the iteration progresses, one HR value $\nu_1^{(l)}$ begins to converge towards $\lambda_1$. However, the term $\left|\nu_1^{(l)} / \lambda_1\right|$ remains large, and the convergence bound still does not fully capture the residual behavior. This corresponds to the plateau in the residual curve, where the influence of $\lambda_1$ is strong but not yet neutralized.

\item \textbf{Phase III:\ HR value accurately captures $\lambda_1$}.  
Once $\nu_1^{(l)} \approx \lambda_1$, we have $\left| s_J^l(\lambda_i) \right| \approx 1$ for all $\lambda_i \in \Lambda_J^c$. The impact of $\lambda_1$ is effectively removed from the bound, and the convergence behavior improves significantly. The residual drops, and the theoretical prediction becomes tight.

\end{itemize}

This analysis reveals a direct link between the evolution of HR values and the shape of the GMRES residual curve.
Each plateau reflects the presence of a small eigenvalue not yet well captured by the Krylov subspace.
Once a small eigenvalue is well approximated, its effect is factored out of the residual bound, and the solver progresses more efficiently. The same mechanism explains the subsequent plateaus and drops: each corresponds to another small eigenvalue of $\mathbf{A}$ being successively captured by HR values.
The more small eigenvalues are well approximated, the smaller the quantity in the convergence bound of \Cref{thm:harmonicRitzValues}, and the more precise the theoretical prediction becomes.

In summary, GMRES convergence slows down when a small eigenvalue of $\mathbf{A}$ is in the process of being resolved by the Krylov subspace.
Once the associated HR value converges, the eigenvalue’s influence fades, and the residual decays more rapidly, as if the eigenvalue were no longer present.


\subsection*{Influence of the deflation}

Having established that the plateaus in the residual history correspond to small eigenvalues of $\mathbf{A}$, we now investigate the effect of deflating the associated eigenvectors on GMRES convergence.
Recall that the eigenvectors of $\mathbf{A}$ approximate those of the Laplace operator with Dirichlet boundary conditions.

We begin by deflating the eigenvector associated with the smallest eigenvalue of $\mathbf{A}$—corresponding to the mode $u_{3,3}$, illustrated in \Cref{fig:cavityPlot}.
In preliminary tests (not shown), we compared two approaches: deflating the discrete eigenvector of $\mathbf{A}$ directly, and deflating the finite element projection of the continuous mode $u_{3,3}$.
Both yielded nearly identical results on fine meshes.
In the remainder of the analysis, we focus on the deflation of projected modes $u_{m,n}$.

\begin{figure}[p]
    \centering
    \begin{subfigure}{0.935\textwidth}
        \raggedright{}
        \subcaption{Number of GMRES iterations}
        \includegraphics[height=0.28\textheight]{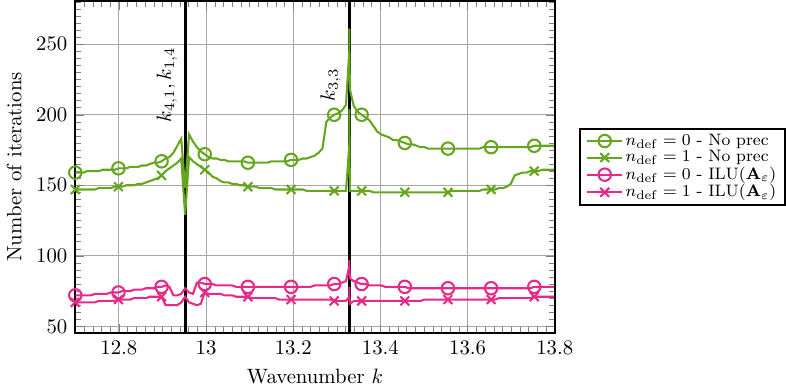}\label{fig:cavityNbIt}
    \end{subfigure}
    \\
    \begin{subfigure}{\textwidth}
        \raggedright{}
        \subcaption{Convergence history without preconditioning}
        \includegraphics[height=0.28\textheight]{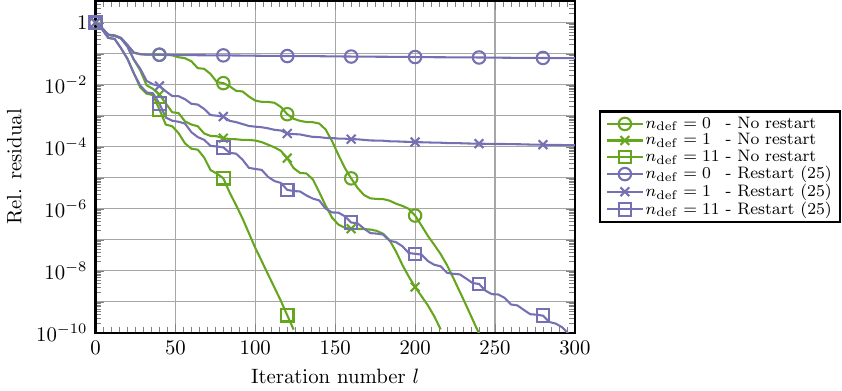}\label{fig:cavityResRestart}
    \end{subfigure} \\
    \begin{subfigure}{\textwidth}
        \raggedright{}
        \subcaption{Convergence history with preconditioning $\mathrm{ILU}(\mathbf{A}_\varepsilon)$}
        \includegraphics[height=0.28\textheight]{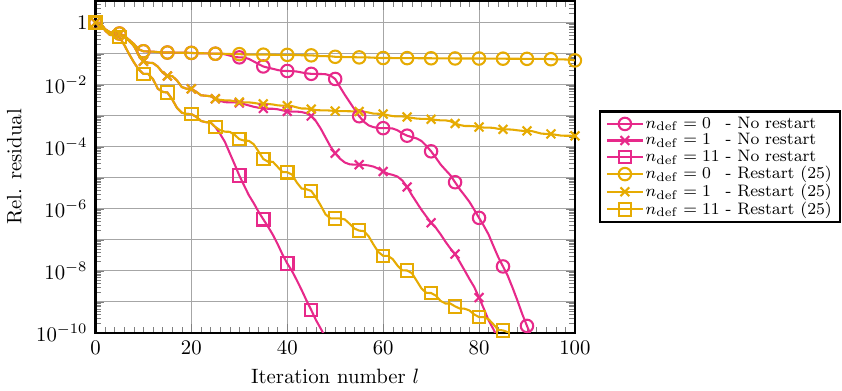}\label{fig:cavityResPrecRestart}
    \end{subfigure}
    \caption{Cavity benchmark. Number of GMRES iteration to reach the relative residual $10^{-6}$, and residual history for $k=\hat{k}$ in different configurations: with/without deflation, with/without preconditioning, with/without restart. $\mathrm{ILU}(\mathbf{A}_\varepsilon)$ refers to the preconditioning by an ILU factorization of the CSL.
    $n_{\text{def}}$ is the number of deflated eigenvalues. $n_{\text{def}}=1$ means deflation of the eigenvalue closest to zero; $n_{\text{def}}=11$ is deflation of all negative eigenvalues.
    Figure 4(a): There are 122 sample points distributed uniformly, and 3 points located around the peak $k_{3,3}$.}
\end{figure}
\Cref{fig:cavityNbIt} shows the effect of deflating $u_{3,3}$ on the number of GMRES iterations as a function of the wavenumber $k$.
Comparing the baseline case without deflation (green $\circ$) and the deflated case (green $\times$), we see that the peak near $k_{3,3}$ disappears when $u_{3,3}$ is deflated, while the peak near $k_{4,1} = k_{1,4}$ remains unaffected.
Note that when $k$ coincides with $k_{3,3}$, there is still a thin peak in the deflated case. The physical problem is ill-posed at this point and numerical resolution is meaningless.
Notably, the iteration count is reduced across the full range of $k$, even for values not directly associated with the deflated mode.

To better understand this effect, we examine the residual history for $k = \hat{k}$, a value close to $k_{3,3}$.
\Cref{fig:cavityTrajRes} revealed that the first plateau in the classic GMRES residual history (green $\circ$) is due to the smallest eigenvalue.
As shown in \Cref{fig:cavityResRestart}, deflating the corresponding eigenvector $u_{3,3}$ (green $\times$) eliminates this first plateau in the residual curve.
The remainder of the curve remains nearly unchanged, as deflation modifies only the contribution of the targeted eigenvector and does not affect the rest of the spectrum.

Finally, we test deflation of the negative part of the spectrum to remove the indefinite part of the problem.
All the eleven negative eigenvalues correspond to Laplace modes whose wavenumbers are below $\hat{k}$ (\Cref{tab:cavityEigval} and~\Cref{fig:cavityPlot}).
This is now a problem with positive eigenvalues, the smallest of which is around $1.5\cdot10^{-2}$ (see \Cref{fig:cavityTrajRes}, upper green line).
The resulting residual history (green $\square $) shows that all plateaus present in the non-deflated case have disappeared.
This confirms the strong connection between small eigenvalues and the observed stagnation behavior, and also highlights the impact of indefiniteness.

\begin{table}
\caption{Cavity benchmark. The smallest eigenvalues of $\mathbf{A}$ and their mode number for $k = \hat{k}$.}\label{tab:cavityEigval}
\centering\small\vspace{-2mm}
\begin{tabular}{|c|cccccccc|}
\hline
    $(n,m)$ &
    $(1,1)$ &
    \begin{tabular}{@{}c@{}} $(1,2)$ \\ $(2,1)$ \end{tabular} &
    $(2,2)$  &
    \begin{tabular}{@{}c@{}} $(1,3)$ \\ $(3,1)$ \end{tabular} &
    \begin{tabular}{@{}c@{}} $(2,3)$ \\ $(3,2)$ \end{tabular} &
    \begin{tabular}{@{}c@{}} $(1,4)$ \\ $(4,1)$ \end{tabular} &
    $(3,3)$ &
    \begin{tabular}{@{}c@{}} $(2,4)$ \\ $(4,2)$ \end{tabular}   \\ \hline
    $\lambda$ & -0.1359 & -0.1101 & -0.0846 & -0.0676 & -0.0425 & -0.0092 & -0.0010 & 0.0154 \\
    \hline
    \end{tabular}
\end{table}


\subsection*{Influence of the preconditioning combined with the deflation}

Building on the previous observations, we now examine the effect of combining preconditioning with deflation.
As detailed in~\Cref{subsec:acceleration}, the system is preconditioned by the ILU decomposition of the CSL matrix.
\Cref{fig:cavityNbIt} shows that, for all wavenumbers, the use of ILU$(\mathbf{A}_\varepsilon)$ as a preconditioner consistently reduces the number of GMRES iterations compared to the unpreconditioned case.
When deflation of $u_{3,3}$ is added to the preconditioning strategy (pink $\times$), the iteration count decreases even further, and the small peak near $k_{3,3}$—previously associated with resonance—is eliminated.

This trend is confirmed by the residual histories in \Cref{fig:cavityResPrecRestart}.
While ILU preconditioning alone significantly shortens the plateaus and accelerates convergence after each stagnation phase, some residual plateaus still remain.
However, deflating the eigenvector associated with the resonant mode $u_{3,3}$ removes the first plateau entirely.
Moreover, when deflation is extended to the eleven eigenvectors corresponding to the eleven negative eigenvalues (\Cref{tab:cavityEigval}), all plateaus vanish, and the residual decays rapidly in a nearly linear fashion.


\subsection*{Influence of the restarted version of GMRES}

Having explored the effects of preconditioning and deflation in full GMRES, we now turn to a more practical variant commonly used in large-scale applications: restarted GMRES.\ 

In the standard implementation of GMRES, the cost per iteration grows with the number of iterations, as each step requires storing and orthogonalizing against an expanding Krylov basis.
To limit memory and computational overhead, GMRES is often restarted every $m$ iterations, retaining only the last Krylov vector to start the next cycle.
This variant, known as GMRES$(m)$, reduces per-iteration cost but sacrifices theoretical guarantees of convergence.

\Cref{fig:cavityResRestart} and~\Cref{fig:cavityResPrecRestart} show the residual histories for GMRES$(m)$ with a restart parameter $m=25$, in both unpreconditioned and ILU-preconditioned settings (blue and yellow curves, respectively).
In both cases, and without deflation, convergence stagnates rapidly and the residual plateaus—even when preconditioning is applied.
We observed similar behavior with larger restart values (e.g., $m=80$, not shown), indicating that the issue is robust.

However, when deflation is introduced, GMRES$(m)$ recovers its ability to converge.
In both the unpreconditioned and preconditioned cases, the residual drops more steadily, and the stagnation phases are significantly reduced or eliminated.
This improvement can be explained as follows: in the restarted setting, the Krylov subspace may not have enough iterations to accurately approximate the small eigenvalues responsible for stagnation.
Deflation explicitly removes these difficult modes, making GMRES$(m)$ more effective.
This highlights an important practical benefit: deflation enhances the robustness of restarted GMRES, particularly near resonant wavenumbers where convergence is otherwise problematic.

%% file: 5_experiment_open.tex
\section{Numerical experiments: scattering problem with quasi-resonances}\label{sec:scatteringProblems}

We now turn to a more challenging configuration: the diffraction of a plane wave by an open cavity. Unlike the previous benchmark, where the system matrix was real, symmetric, and associated with exact resonance modes, the matrix $\mathbf{A}$ here is complex and non-normal. Although exact resonances are absent, the problem features \emph{quasi-modes}, which can lead to similar convergence difficulties for iterative solvers. As before, we use numerical experiments to study GMRES convergence and the effect of acceleration techniques.

\subsection*{Description of the benchmark}

The problem consists of the scattering of a time-harmonic plane wave $
    u_\mathrm{inc}(x,y) = e^{\iu k(\cos(\theta) x + \sin(\theta) y)}$ by an obstacle $\mathcal{O}$ in free space $\mathbb{R}^2$, where $\theta$ is the incident angle. The boundary of the obstacle is assigned either a Dirichlet condition (Dirichlet case) or a Neumann condition (Neumann case).

To model the unbounded domain, we truncate the computational domain to a square region $\Omega_\mathrm{dom} = {(-L, L)}^2 \setminus \mathcal{O}$ and surround it with perfectly matched layers (PMLs) of thickness $L_\mathrm{pml}$, denoted $\Omega_\mathrm{pml}$~\cite{BermudezHervellaNietoEtAl2004EBP,BermudezHervellaNietoEtAl2007OPM}.
The PMLs absorb outgoing waves and eliminate reflections from the artificial boundary, on which a homogeneous Dirichlet condition is imposed. The PML formulation introduces complex coordinate stretchings:
\[
    \gamma_x(x) := 1 + \frac{i\sigma_x(x)}{k}, \quad \gamma_y(y) := 1 + \frac{i\sigma_y(y)}{k},
\]
with absorption profiles:
\[
    \sigma_x(x) := \frac{\mathbbm{1}_\mathrm{pml}(x)}{L_\mathrm{pml} - |x| + L}, \quad
    \sigma_y(y) := \frac{\mathbbm{1}_\mathrm{pml}(y)}{L_\mathrm{pml} - |y| + L}.
\]
The scattered field $u$ then satisfies:
\begin{align*}
    \left\{
    \begin{aligned}
        \frac{\partial}{\partial x} \left( \frac{\gamma_y}{\gamma_x} \frac{\partial u}{\partial x} \right) +
        \frac{\partial}{\partial y} \left( \frac{\gamma_x}{\gamma_y} \frac{\partial u}{\partial y} \right) +
        \gamma_x \gamma_y k^2 u & = 0 \quad & \text{in } \Omega_\mathrm{dom} \cup \Omega_\mathrm{pml}, \\
        u = -u_\mathrm{inc} \quad \text{or} \quad
        \partial_{\boldsymbol{n}} u & = -\partial_{\boldsymbol{n}} u_\mathrm{inc} \quad & \text{on } \Gamma_\mathrm{obs}, \\
        u & = 0 \quad & \text{on } \Gamma_\mathrm{ext}.
    \end{aligned}
    \right.
\end{align*}

The obstacle $\mathcal{O}$ is an open rectangular cavity of length $L_\mathcal{O}$ and opening width $l_\mathcal{O}$ (see \Cref{fig:scatteringPML}).
In contrast to closed domains, this setup does not admit true resonance modes but supports \emph{quasi-modes}, i.e., modes that approximately satisfy the homogeneous equation and are spatially localized. 
This configuration is a case of a weak-trapping (and, more specifically, parabolic)~\cite{ChandlerWildeSpenceEtAl2020HFB}. The decay of \(\epsilon(k_j)\) in~\Cref{def:quasimodes} is therefore polynomial.

In the Dirichlet case, quasi-modes are known to accumulate near wavenumbers $\pi n / l_\mathcal{O}$ for $n \in \mathbb{N}$~\cite[Theorem 5.1]{ChandlerWildeGrahamEtAl2009CNE}.
However, numerical simulations indicate the presence of additional quasi-modes at other wavenumbers.
These are often close to the resonance frequencies of a closed rectangular cavity with the same dimensions.
In that case, the Dirichlet resonance wavenumbers are given by
\[
    k_{n,m} = \pi \sqrt{\frac{m^2}{L_\mathcal{O}^2} + \frac{n^2}{l_\mathcal{O}^2}}, \quad m,n \in \mathbb{N}.
\]
For the Neumann case (used in our experiments), the expected quasi-modes are close to the resonance frequencies of a closed rectangular cavity with Neumann boundary conditions on all sides except the left edge (Dirichlet), given by
\[
    k_{n,m} = \pi \sqrt{\frac{{(m + 1/2)}^2}{L_\mathcal{O}^2} + \frac{n^2}{l_\mathcal{O}^2}}, \quad m \in \N \cup \{0\}, \, n \in \N.
\]
In both cases, the quasi-modes inside the cavity resemble the eigenmodes of the corresponding closed configuration.

After discretization using finite elements, the problem leads to a linear system $\mathbf{A} \mathbf{u} = \mathbf{b}$.
The resulting matrix $\mathbf{A}$ is complex, non-normal, and has a spectrum that includes complex-valued eigenvalues.

All numerical results presented in this section correspond to the Neumann case.
The Dirichlet case shows similar behavior and is omitted for brevity.
We use the following parameters $
    L_\mathcal{O} = 1.3, \quad l_\mathcal{O} = 0.4, \quad \theta = \frac{4\pi}{10}, \quad \mathsf{h} = \frac{1}{20}, \quad \text{$P_3$ continuous elements}, \quad \mathsf{tol} = 10^{-6}, \quad D_\lambda \geq 30$.
This configuration leads to approximately 13,500 degrees of freedom.

\begin{figure}
    \centering
    \includegraphics[width=0.4\textwidth]{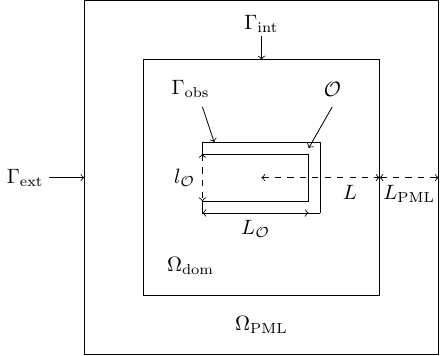}
    \caption{Scattering benchmark. Computational domain with PML.}\label{fig:scatteringPML}
\end{figure}

As a preliminary study, we have studied the non-normality of $\mathbf{A}$.
The condition number of $\mathbf{V}$ is $6.389\cdot 10^2$, which is small compared to highly non-normal problems found in the literature \cite{Embree2022HDA}.
We observed that the non-normality increases with the frequency (not shown here for sake of brevity), and that the condition numbers of the eigenvalues associated with physical modes are close to $1$.
It suggests that, in this benchmark, quasi-resonances do not play a significant role in the non-normality of $\mathbf{A}$.
A complete study of this question is out of the scope of this work.

\begin{figure}
    \centering
    \begin{subfigure}{\textwidth}
        \raggedright{}
        \subcaption{Number of GMRES iterations}
        \includegraphics[height=0.28\textheight]{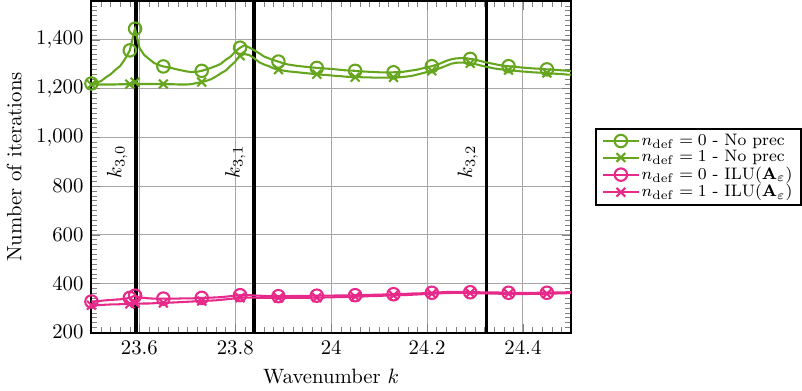}\label{fig:openCavityNbit}
    \end{subfigure} \\
    \begin{subfigure}{\textwidth}
        \raggedright{}
        \subcaption{Convergence history without preconditioning}
        \includegraphics[height=0.28\textheight]{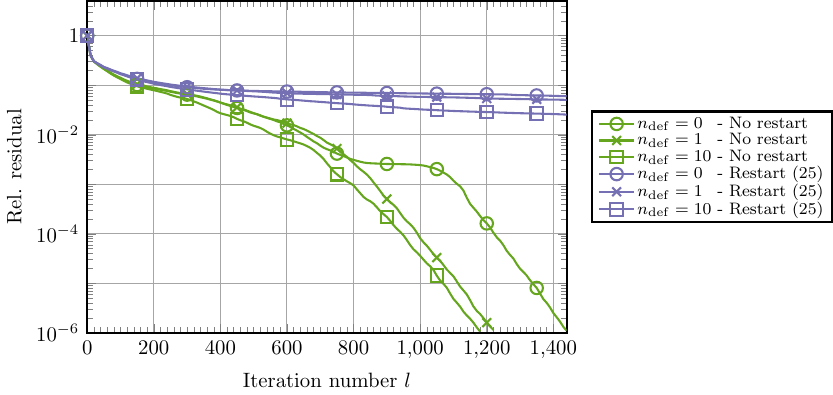}\label{fig:openCavityResRestart}
    \end{subfigure} \\
    \bigskip
    \begin{subfigure}{\textwidth}
        \raggedright{}
        \caption{Convergence history with preconditioning $\mathrm{ILU}(\mathbf{A}_\varepsilon)$}
        \includegraphics[height=0.28\textheight]{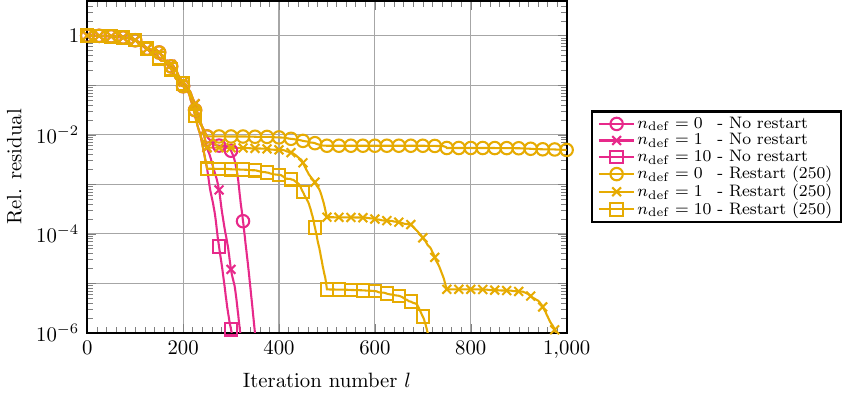}\label{fig:openCavityResPrecRestart}
    \end{subfigure}
    \caption{Scattering benchmark. Number of GMRES iteration to reach the relative residual $10^{-6}$, and residual history for $k=\hat{k}$ in different configurations: with/without deflation, with/without preconditioning, with/without restart. $\mathrm{ILU}(\mathbf{A}_\varepsilon)$ refers to the preconditioning by an ILU factorization of the CSL.
$n_{\text{def}}$ is the number of deflated eigenvalues. $n_{\text{def}}=1$ means deflation of the eigenvalue closest to zero; $n_{\text{def}}=10$ is deflation of the closest quasi-modes.
Figure 6(a): There are 100 sample points distributed uniformly, and 3 points located around the peak $k_{3,0}$.}
\end{figure}


\subsection*{Influence of the wavenumber on the number of iterations}

We now investigate how the wavenumber $k$ affects the convergence of GMRES for the open cavity problem. \Cref{fig:openCavityNbit} shows the number of GMRES iterations required to reach a residual tolerance of $10^{-6}$ as a function of $k \in [23.5, 24.5]$ (green $\circ$).

As in the closed cavity benchmark, we observe peaks in the iteration count near the expected wavenumbers of quasi-modes, specifically around $k_{3,0} \approx 23.591$, $k_{3,1}$, and $k_{3,2}$.
However, the peaks here are less pronounced than in the cavity benchmark, reflecting the fact that quasi-modes—unlike true resonance modes—do not make the problem singular, but still degrade the performance of iterative solvers.
Among the peaks, the one near $k_{3,0}$ is the most prominent.
To better understand the behavior of GMRES in this regime, we focus in the following on the wavenumber $\hat{k} := 23.591$.


\subsection*{Influence of the wavenumber on the residual history}

\begin{figure}[htb!]
    \centering
    \includegraphics[width=0.85\textwidth]{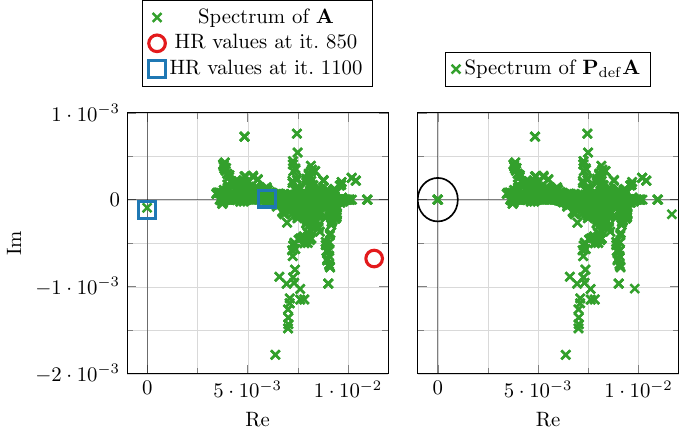}
    \caption{Scattering benchmark. HR values and deflated spectrum. The large black circle around the eigenvalue at the origin indicates that this eigenvalue is deflated.}\label{fig:openCavityspectrum}
\end{figure}

We now examine the convergence behavior near the wavenumber $\hat{k} = 23.591$, where the influence of a quasi-mode was observed in the iteration count.

\Cref{fig:openCavityResRestart} shows the GMRES residual history without any acceleration technique.
Compared to the cavity problem, convergence is significantly slower.
Only a single plateau is observed, between iterations 800 and 1050, after which the residual begins to decrease more rapidly.
As in the previous benchmark, this stagnation can be explained using HR values and \Cref{thm:harmonicRitzValues}.

Since the matrix $\mathbf{A}$ has a complex spectrum, we examine the distribution of HR values in the complex plane.
\Cref{fig:openCavityspectrum} (left) compares the positions of HR values at iteration 850—during the plateau—and iteration 1100—just after the plateau ends.
At iteration 850, we observe that an HR value is approaching the smallest eigenvalue $\lambda_1$ of $\mathbf{A}$, located in the lower right region of the complex plane.
By iteration 1100, this eigenvalue is well approximated: its condition number is removed from the sum in the bound of \Cref{thm:harmonicRitzValues}, and $\max_{\lambda_i \in \Lambda_J^c} \abs{s_J^{l}(\lambda_i)}\approx 1$.
This coincides with the end of the plateau and the beginning of the fast decay.

Although the spectrum is now complex, the convergence behavior remains consistent with the interpretation provided by \Cref{thm:harmonicRitzValues}.
In this benchmark, the slower convergence of HR values towards the smallest eigenvalue explains the extended duration of the plateau phase.


\subsection*{Influence of the deflation}
To mitigate the slow convergence near the quasi-mode, we now assess the impact of deflation on the solver’s performance.

Although no analytical expression is available for the quasi-modes in this configuration, it is possible to construct approximate deflation vectors based on physical intuition.
We use eigenvectors of the corresponding closed cavity problem—with Dirichlet conditions on the left edge and Neumann conditions on the remaining edges—extended by zero outside the cavity.

\Cref{fig:openCavityspectrum} (right) displays the spectrum of the deflated matrix $\mathbf{P}_\mathrm{def} \mathbf{A}$.
The large black circle around the origin indicates that the eigenvalue associated with the quasi-mode is reduced to zero, confirming that the deflated vector effectively captures the problematic mode.

The influence of this deflation on convergence is shown in \Cref{fig:openCavityNbit} (green $\times$).
As expected, the peak near $k_{3,0}$—previously linked to the quasi-mode—is entirely removed.
Moreover, the number of GMRES iterations is slightly reduced across other values of $k$, even away from the targeted quasi-mode.

To explore this further, we compare the residual histories with and without deflation at $k = \hat{k}$, near $k_{3,0}$ (\Cref{fig:openCavityResRestart}).
The plateau observed in the unaccelerated case disappears with deflation.
However, between iterations 600 and 800, the residual with deflation is briefly higher than in the non-deflated case.
This is likely due to the fact that the deflated vector only approximates the true quasi-mode.

We also test the deflation of multiple approximate quasi-modes (green $\square$), but the improvement in convergence is modest, suggesting diminishing returns beyond the first key mode.


\subsection*{Influence of preconditioning combined with deflation}

We now assess how ILU preconditioning interacts with deflation in improving GMRES convergence.

As shown in \Cref{fig:openCavityNbit}, applying the ILU$(\mathbf{A}_\varepsilon)$ preconditioner systematically reduces the number of iterations required to reach the convergence tolerance, across all wavenumbers.
Although the benefit of combining preconditioning with deflation is not visually striking at the scale of the plot, the effect is measurable.
At $k = \hat{k}$, for example, the number of iterations drops from 350 (preconditioning only) to 319 when deflation is added—an improvement of approximately 9\%.

The residual histories in \Cref{fig:openCavityResPrecRestart} provide a more detailed view.
With preconditioning alone (pink $\circ$), the plateau associated with the quasi-mode persists, though it is shorter and followed by a sharp decline in the residual.
When deflation of the approximate quasi-mode is added, the plateau is eliminated entirely, and convergence becomes both faster and smoother.

These results confirm that ILU$(\mathbf{A}_\varepsilon)$ is effective in mitigating spectral clustering effects and improving convergence.
Nonetheless, the combination of preconditioning and deflation remains the most effective strategy, even when the deflated vectors are only approximations of the true quasi-modes.


\subsection*{Influence of the restarted version of GMRES}

We now turn to the restarted GMRES$(m)$, to evaluate its behavior in this more challenging setting. As with the cavity problem, restarting is used to reduce the memory footprint and per-iteration cost.
\Cref{fig:openCavityResRestart} (for $m = 25$) and~\Cref{fig:openCavityResPrecRestart} (for $m = 250$) display the residual histories with and without preconditioning.

Without deflation, we note that GMRES$(m)$ fails to converge—regardless of preconditioning—even when the restart value is increased (e.g., $m = 300$, not shown).
This highlights the difficulty of resolving the quasi-mode and overcoming the effects of non-normality and indefiniteness with restarts alone.

In contrast to the cavity benchmark, deflation alone (blue $\times$ and $\square$) is not sufficient to guarantee convergence within a reasonable number of iterations.
This can be attributed to the compounded challenges of non-normality, indefiniteness, and the complexity of the quasimodal structure in this benchmark.

Only the combination of deflation and preconditioning (yellow $\times$) yields successful convergence, though a high restart value is required: for $m \leq 250$, the tolerance is not reached within 2000 iterations.
However, increasing the number of deflated vectors (yellow $\square$) improves convergence significantly, making it possible to reduce the restart parameter while still achieving convergence (detailed results not shown here). These results underscore the importance of using both preconditioning and deflation in restarted settings, especially when dealing with non-normal systems influenced by quasi-modes.

%% file: 6_conclusion.tex
\section{Conclusion}\label{sec:conclusion}

This study examined the convergence behavior of GMRES for the finite element solution of Helmholtz problems, emphasizing the impact of resonance and quasi-modes.

Novel convergence bounds based on harmonic Ritz values have been proposed for general systems.
These bounds are able to describe nonlinear residual histories, offering a detailed explanation of residual plateaus.
Through numerical experiments on both closed and open cavity benchmarks, we observed that convergence stalls align with wavenumbers associated with resonances or quasi-modes.
These stalls were accurately related to the evolution of harmonic Ritz values.

The main observations of the numerical studies can be summarized as follows:
\begin{itemize}
    \setlength{\itemsep}{0pt}
    \setlength{\parskip}{0pt}
    \setlength{\parsep}{0pt}
    \item GMRES convergence is highly slowed and plateaus because of small eigenvalues due to resonances or quasi-resonances. In particular, restarted GMRES cannot achieve convergence in some cases.
    \item Deflation of the eigenvectors corresponding to resonances or quasi-resonances, even when approximated, removes plateaus and accelerates convergence.
    This approach requires knowledge of the eigenvectors and slightly increases the cost per iteration, which depends on the number of the deflated eigenvectors.
    
    \item Although deflation and preconditioning techniques improved convergence when used individually, combining them proved to be the most effective approach. For open domain problems, ILU$(\mathbf{A}_\varepsilon)$ preconditioning enhanced performance, albeit at the cost of an additional matrix product. However, the convergence of restarted GMRES was only obtained when used in conjunction with deflation.
\end{itemize}

These results underscore the value of combining an a priori knowledge of spectral properties with tailored numerical techniques to design robust solvers for high-frequency wave propagation problems.
Future work includes extending these strategies to 3D configurations and exploring recent preconditioning techniques based on domain decomposition.
A more detailed study of the factors of the convergence bounds that account for the non-normality of the problem, such as the sum of the eigenvalue condition numbers, would also be of interest.

%% file: bibliography.bib
@Article{BabuskaSauter1997IPE,
  author    = {Babuška, Ivo M. and Sauter, Stefan A.},
  journal   = {SIAM Journal on Numerical Analysis},
  title     = {Is the pollution effect of the {FEM} avoidable for the {H}elmholtz equation considering high wave numbers?},
  year      = {1997},
  issn      = {1095-7170},
  month     = dec,
  number    = {6},
  pages     = {2392--2423},
  volume    = {34},
  doi       = {10.1137/s0036142994269186},
  publisher = {Society for Industrial & Applied Mathematics (SIAM)},
}

@Article{BastianWittumEtAl1998AMM,
  author    = {Bastian, P. and Wittum, G. and Hackbusch, W.},
  journal   = {Computing},
  title     = {Additive and multiplicative multi-grid — A comparison},
  year      = {1998},
  issn      = {1436-5057},
  month     = dec,
  number    = {4},
  pages     = {345--364},
  volume    = {60},
  doi       = {10.1007/bf02684380},
  publisher = {Springer Science and Business Media LLC},
}

@Article{BecacheDhiaEtAl2006PML,
  author    = {Bécache, E. and Dhia, A. S. Bonnet-Ben and Legendre, G.},
  journal   = {SIAM Journal on Numerical Analysis},
  title     = {Perfectly Matched Layers for Time-Harmonic Acoustics in the Presence of a Uniform Flow},
  year      = {2006},
  issn      = {1095-7170},
  month     = jan,
  number    = {3},
  pages     = {1191--1217},
  volume    = {44},
  doi       = {10.1137/040617741},
  publisher = {Society for Industrial & Applied Mathematics (SIAM)},
}

@Article{BeckermannGoreinovEtAl2005SRE,
  author    = {B. Beckermann and S. A. Goreinov and E. E. Tyrtyshnikov},
  journal   = {{SIAM} Journal on Matrix Analysis and Applications},
  title     = {Some Remarks on the {E}lman Estimate for {GMRES}},
  year      = {2005},
  month     = {jan},
  number    = {3},
  pages     = {772--778},
  volume    = {27},
  doi       = {10.1137/040618849},
  publisher = {Society for Industrial {\&} Applied Mathematics ({SIAM})},
}

@Article{Berenger1994PML,
  author    = {Berenger, Jean-Pierre},
  journal   = {Journal of Computational Physics},
  title     = {A perfectly matched layer for the absorption of electromagnetic waves},
  year      = {1994},
  issn      = {0021-9991},
  month     = oct,
  number    = {2},
  pages     = {185--200},
  volume    = {114},
  doi       = {10.1006/jcph.1994.1159},
  publisher = {Elsevier BV},
}

@Article{BeriotModave2020APM,
  author    = {Hadrien B{\'{e}}riot and Axel Modave},
  journal   = {International Journal for Numerical Methods in Engineering},
  title     = {An automatic perfectly matched layer for acoustic finite element simulations in convex domains of general shape},
  volume = {122},
  number = {5},
  pages = {1239-1261},
  year      = {2020},
  month     = {nov},
  doi       = {10.1002/nme.6560},
  publisher = {Wiley},
}

@Article{BermudezHervellaNietoEtAl2004EBP,
  author    = {Alfredo Berm{\'{u}}dez and L. Hervella-Nieto and A. Prieto and R. Rodr{\'{i}}guez},
  journal   = {Comptes Rendus Mathematique},
  title     = {An exact bounded {PML} for the {H}elmholtz equation},
  year      = {2004},
  month     = {dec},
  number    = {11},
  pages     = {803--808},
  volume    = {339},
  doi       = {10.1016/j.crma.2004.10.006},
  publisher = {Elsevier {BV}},
}

@Article{BermudezHervellaNietoEtAl2007OPM,
  author    = {A. Berm{\'{u}}dez and L. Hervella-Nieto and A. Prieto and R. Rodr{\'{i}}guez},
  journal   = {Journal of Computational Physics},
  title     = {An optimal perfectly matched layer with unbounded absorbing function for time-harmonic acoustic scattering problems},
  year      = {2007},
  month     = {may},
  number    = {2},
  pages     = {469--488},
  volume    = {223},
  doi       = {10.1016/j.jcp.2006.09.018},
  publisher = {Elsevier {BV}},
}

@Article{CampbellIpsenEtAl1996GMP,
  author    = {S. L. Campbell and I. C. F. Ipsen and C. T. Kelley and C. D. Meyer},
  journal   = {{BIT} Numerical Mathematics},
  title     = {{GMRES} and the minimal polynomial},
  year      = {1996},
  month     = {dec},
  number    = {4},
  pages     = {664--675},
  volume    = {36},
  doi       = {10.1007/bf01733786},
  publisher = {Springer Science and Business Media {LLC}},
}

@Article{Cao1997NCB,
  author    = {Zhi-Hao Cao},
  journal   = {Applied Numerical Mathematics},
  title     = {A note on the convergence behavior of {GMRES}},
  year      = {1997},
  month     = oct,
  number    = {1},
  pages     = {13--20},
  volume    = {25},
  doi       = {10.1016/s0168-9274(97)00039-1},
  publisher = {Elsevier {BV}},
}

@Article{ChandlerWildeSpenceEtAl2020HFB,
  author    = {S. N. Chandler-Wilde and E. A. Spence and A. Gibbs and V. P. Smyshlyaev},
  journal   = {{SIAM} Journal on Mathematical Analysis},
  title     = {High-frequency Bounds for the {H}elmholtz Equation Under Parabolic Trapping and Applications in Numerical Analysis},
  year      = {2020},
  month     = {jan},
  number    = {1},
  pages     = {845--893},
  volume    = {52},
  doi       = {10.1137/18m1234916},
  publisher = {Society for Industrial {\&} Applied Mathematics ({SIAM})},
}

@Article{ChandlerWildeGrahamEtAl2009CNE,
  author    = {S. N. Chandler-Wilde and I. G. Graham and S. Langdon and M. Lindner},
  journal   = {Journal of Integral Equations and Applications},
  title     = {Condition number estimates for combined potential boundary integral operators in acoustic scattering},
  year      = {2009},
  month     = {jun},
  pages={229--279},
  number    = {2},
  volume    = {21},
  doi       = {10.1216/jie-2009-21-2-229},
  publisher = {Rocky Mountain Mathematics Consortium},
}

@Article{CoolsVanroose2013LFA,
  author    = {Cools, Siegfried and Vanroose, Wim},
  journal   = {Numerical Linear Algebra with Applications},
  title     = {Local {F}ourier analysis of the complex shifted {L}aplacian preconditioner for {H}elmholtz problems},
  year      = {2013},
  issn      = {1099-1506},
  month     = apr,
  number    = {4},
  pages     = {575--597},
  volume    = {20},
  doi       = {10.1002/nla.1881},
  publisher = {Wiley},
}

@Article{CrouzeixPalencia2017NRI,
  author    = {M. Crouzeix and C. Palencia},
  journal   = {{SIAM} Journal on Matrix Analysis and Applications},
  title     = {The Numerical Range is a $(1+\sqrt{2})$-Spectral Set},
  year      = {2017},
  month     = {jan},
  number    = {2},
  pages     = {649--655},
  volume    = {38},
  doi       = {10.1137/17m1116672},
  publisher = {Society for Industrial {\&} Applied Mathematics ({SIAM})},
}

@Article{VorstVuik1993SCB,
  author    = {van der Vorst, H. A. and Vuik, C.},
  journal   = {Journal of Computational and Applied Mathematics},
  title     = {The superlinear convergence behaviour of {GMRES}},
  year      = {1993},
  month     = nov,
  number    = {3},
  pages     = {327--341},
  volume    = {48},
  doi       = {10.1016/0377-0427(93)90028-a},
  publisher = {Elsevier {BV}},
}

@Book{DoleanJolivetEtAl2015IDD,
  author    = {Dolean, Victorita and Jolivet, Pierre and Nataf, Frédéric},
  publisher = {Society for Industrial and Applied Mathematics},
  title     = {An Introduction to Domain Decomposition Methods: Algorithms, Theory, and Parallel Implementation},
  year      = {2015},
  isbn      = {9781611974065},
  month     = nov,
  doi       = {10.1137/1.9781611974065},
}

@Article{DoleanFryEtAl2025AWR,
  author        = {Dolean, Victorita and Fry, Mark and Langer, Matthias and Parolin, Emile and Tournier, Pierre-Henri},
  title         = {Achieving wavenumber robustness in domain decomposition for heterogeneous {H}elmholtz equation: an overview of spectral coarse spaces},
  year          = {2025},
  month         = sep,
  archiveprefix = {arXiv},
  copyright     = {arXiv.org perpetual, non-exclusive license},
  doi           = {10.48550/ARXIV.2509.02131},
  eprint        = {2509.02131},
  primaryclass  = {math.NA},
  publisher     = {arXiv},
  note          = {arXiv preprint, arXiv.2509.02131},
}

@Book{DuffErismanEtAl2017DMS,
  author    = {Duff, I. S. and Erisman, A. M. and Reid, J. K.},
  publisher = {Oxford University PressOxford},
  title     = {Direct Methods for Sparse Matrices},
  year      = {2017},
  isbn      = {9780191746420},
  month     = jan,
  doi       = {10.1093/acprof:oso/9780198508380.001.0001},
}

@Article{DwarkaVuik2022SML,
  author    = {Dwarka, Vandana and Vuik, Cornelis},
  journal   = {Journal of Computational Physics},
  title     = {Scalable multi-level deflation preconditioning for highly indefinite time-harmonic waves},
  year      = {2022},
  issn      = {0021-9991},
  month     = nov,
  pages     = {111327},
  volume    = {469},
  doi       = {10.1016/j.jcp.2022.111327},
  publisher = {Elsevier BV},
}

@Misc{Embree2022HDA,
  author    = {Embree, Mark},
  note      = {arXiv preprint, arXiv:2209.01231},
  title     = {How Descriptive are {GMRES} Convergence Bounds?},
  year      = {2022},
  copyright = {arXiv.org perpetual, non-exclusive license},
  doi       = {10.48550/ARXIV.2209.01231},
  publisher = {arXiv},
  url       = {https://arxiv.org/abs/2209.01231},
}

@Article{Embree2025EEB,
  author    = {Embree, Mark},
  journal   = {Linear Algebra and its Applications},
  title     = {Extending {E}lman’s bound for {GMRES}},
  year      = {2025},
  issn      = {0024-3795},
  month     = dec,
  pages     = {54--70},
  volume    = {726},
  doi       = {10.1016/j.laa.2025.07.007},
  publisher = {Elsevier BV},
}

@Article{EngquistMajda1977ABC,
  author    = {Engquist, Björn and Majda, Andrew},
  journal   = {Proceedings of the National Academy of Sciences},
  title     = {Absorbing boundary conditions for numerical simulation of waves},
  year      = {1977},
  issn      = {1091-6490},
  month     = may,
  number    = {5},
  pages     = {1765--1766},
  volume    = {74},
  doi       = {10.1073/pnas.74.5.1765},
  publisher = {Proceedings of the National Academy of Sciences},
}

@Article{Erlangga2007AIM,
  author    = {Yogi A. Erlangga},
  journal   = {Archives of Computational Methods in Engineering},
  title     = {Advances in Iterative Methods and Preconditioners for the {H}elmholtz Equation},
  year      = {2007},
  month     = {dec},
  number    = {1},
  pages     = {37--66},
  volume    = {15},
  doi       = {10.1007/s11831-007-9013-7},
  publisher = {Springer Science and Business Media {LLC}},
}

@Article{ErlanggaNabben2008DBP,
  author    = {Erlangga, Yogi A. and Nabben, Reinhard},
  journal   = {SIAM Journal on Matrix Analysis and Applications},
  title     = {Deflation and Balancing Preconditioners for {K}rylov Subspace Methods Applied to Nonsymmetric Matrices},
  year      = {2008},
  issn      = {1095-7162},
  month     = jan,
  number    = {2},
  pages     = {684--699},
  volume    = {30},
  doi       = {10.1137/060678257},
  publisher = {Society for Industrial \& Applied Mathematics (SIAM)},
}

@Article{ErlanggaOosterleeEtAl2006NMB,
  author    = {Erlangga, Y. A. and Oosterlee, C. W. and Vuik, C.},
  journal   = {SIAM Journal on Scientific Computing},
  title     = {A Novel Multigrid Based Preconditioner For Heterogeneous {H}elmholtz Problems},
  year      = {2006},
  issn      = {1095-7197},
  month     = jan,
  number    = {4},
  pages     = {1471--1492},
  volume    = {27},
  doi       = {10.1137/040615195},
  publisher = {Society for Industrial & Applied Mathematics (SIAM)},
}

@Article{ErlanggaVuikEtAl2006CMI,
  author    = {Erlangga, Y.A. and Vuik, C. and Oosterlee, C.W.},
  journal   = {Applied Numerical Mathematics},
  title     = {Comparison of multigrid and incomplete {LU} shifted-{L}aplace preconditioners for the inhomogeneous {H}elmholtz equation},
  year      = {2006},
  issn      = {0168-9274},
  month     = may,
  number    = {5},
  pages     = {648--666},
  volume    = {56},
  doi       = {10.1016/j.apnum.2005.04.039},
  publisher = {Elsevier BV},
}

@InBook{ErnstGander2012WII,
  author    = {Ernst, O. G. and Gander, M. J.},
  editor    = {Graham, Ivan G. and Hou, Thomas Y. and Lakkis, Omar and Scheichl, Robert},
  pages     = {325--363},
  publisher = {Springer Berlin Heidelberg},
  title     = {Why it is Difficult to Solve Helmholtz Problems with Classical Iterative Methods},
  year      = {2012},
  isbn      = {9783642220616},
  month     = aug,
  booktitle = {Numerical Analysis of Multiscale Problems},
  doi       = {10.1007/978-3-642-22061-6_10},
  issn      = {1439-7358},
  url       = {https://doi.org/10.1007/978-3-642-22061-6_10},
}

@Article{FrankVuik2001CDB,
  author    = {Frank, J. and Vuik, C.},
  journal   = {SIAM Journal on Scientific Computing},
  title     = {On the Construction of Deflation-Based Preconditioners},
  year      = {2001},
  issn      = {1095-7197},
  month     = jan,
  number    = {2},
  pages     = {442--462},
  volume    = {23},
  doi       = {10.1137/s1064827500373231},
  publisher = {Society for Industrial \& Applied Mathematics (SIAM)},
}

@Article{GalkowskiMarchandEtAl2021ETH,
  author    = {Galkowski, Jeffrey and Marchand, Pierre and Spence, Euan A.},
  journal   = {SIAM Journal on Mathematical Analysis},
  title     = {Eigenvalues of the Truncated {H}elmholtz Solution Operator under Strong Trapping},
  year      = {2021},
  issn      = {1095-7154},
  month     = jan,
  number    = {6},
  pages     = {6724--6770},
  volume    = {53},
  doi       = {10.1137/21m1399658},
  publisher = {Society for Industrial & Applied Mathematics (SIAM)},
}

@Article{GanderGrahamEtAl2015AGH,
  author    = {M. J. Gander and I. G. Graham and E. A. Spence},
  journal   = {Numerische Mathematik},
  title     = {Applying {GMRES} to the {H}elmholtz equation with shifted {L}aplacian preconditioning: what is the largest shift for which wavenumber-independent convergence is guaranteed?},
  year      = {2015},
  month     = jan,
  number    = {3},
  pages     = {567--614},
  volume    = {131},
  doi       = {10.1007/s00211-015-0700-2},
  publisher = {Springer Science and Business Media {LLC}},
}

@Article{GarciaRamosKehlEtAl2020PDM,
  author    = {García Ramos, Luis and Kehl, René and Nabben, Reinhard},
  journal   = {SIAM Journal on Matrix Analysis and Applications},
  title     = {Projections, Deflation, and Multigrid for Nonsymmetric Matrices},
  year      = {2020},
  issn      = {1095-7162},
  month     = jan,
  number    = {1},
  pages     = {83--105},
  volume    = {41},
  doi       = {10.1137/18m1180268},
  publisher = {Society for Industrial \& Applied Mathematics (SIAM)},
}

@Article{MacLachlanOosterlee2008AMS,
  author    = {MacLachlan, Scott P. and Oosterlee, Cornelis W.},
  journal   = {SIAM Journal on Scientific Computing},
  title     = {Algebraic Multigrid Solvers for Complex-Valued Matrices},
  year      = {2008},
  issn      = {1095-7197},
  month     = jan,
  number    = {3},
  pages     = {1548--1571},
  volume    = {30},
  doi       = {10.1137/070687232},
  publisher = {Society for Industrial & Applied Mathematics (SIAM)},
}

@Article{MarchandGalkowskiEtAl2022AGH,
  author    = {Marchand, P. and Galkowski, J. and Spence, E. A. and Spence, A.},
  journal   = {Advances in Computational Mathematics},
  title     = {Applying {GMRES} to the {H}elmholtz equation with strong trapping: how does the number of iterations depend on the frequency?},
  year      = {2022},
  pages={37},
  issn      = {1572-9044},
  month     = jun,
  number    = {4},
  volume    = {48},
  doi       = {10.1007/s10444-022-09931-9},
  publisher = {Springer Science and Business Media LLC},
}

@Article{MarchnerBeriotEtAl2021SPM,
  author    = {Marchner, Philippe and Bériot, Hadrien and Antoine, Xavier and Geuzaine, Christophe},
  journal   = {Journal of Computational Physics},
  title     = {Stable Perfectly Matched Layers with {L}orentz transformation for the convected {H}elmholtz equation},
  year      = {2021},
  issn      = {0021-9991},
  month     = may,
  pages     = {110180},
  volume    = {433},
  doi       = {10.1016/j.jcp.2021.110180},
  publisher = {Elsevier BV},
}

@Article{ModaveGeuzaineEtAl2020CTH,
  author    = {A. Modave and C. Geuzaine and X. Antoine},
  journal   = {Journal of Computational Physics},
  title     = {Corner treatments for high-order local absorbing boundary conditions in high-frequency acoustic scattering},
  year      = {2020},
  month     = {jan},
  pages     = {109029},
  volume    = {401},
  doi       = {10.1016/j.jcp.2019.109029},
  publisher = {Elsevier {BV}},
}

@phdthesis{Moitier2019EME,
  author      = {Moitier, Zo{\"i}s},
  month       = Oct,
  title       = {{{\'E}tude math{\'e}matique et num{\'e}rique des r{\'e}sonances dans une micro-cavit{\'e} optique}},
  year        = {2019},
  hal_id      = {tel-02308978},
  hal_version = {v2},
  number      = {2019REN1S053},
  school      = {{Universit{\'e} de Rennes}},
  url         = {https://theses.hal.science/tel-02308978}
}

@Article{NabbenVuik2008CAV,
  author    = {Nabben, R. and Vuik, C.},
  journal   = {Numerical Linear Algebra with Applications},
  title     = {A comparison of abstract versions of deflation, balancing and additive coarse grid correction preconditioners},
  year      = {2008},
  issn      = {1099-1506},
  month     = jan,
  number    = {4},
  pages     = {355--372},
  volume    = {15},
  doi       = {10.1002/nla.571},
  publisher = {Wiley},
}

@Article{PetersenDreyerEtAl2006AFS,
  author    = {Petersen, Steffen and Dreyer, Daniel and von Estorff, Otto},
  journal   = {Computer Methods in Applied Mechanics and Engineering},
  title     = {Assessment of finite and spectral element shape functions for efficient iterative simulations of interior acoustics},
  year      = {2006},
  issn      = {0045-7825},
  month     = sep,
  number    = {44–47},
  pages     = {6463--6478},
  volume    = {195},
  doi       = {10.1016/j.cma.2006.01.008},
  publisher = {Elsevier BV},
}

@Misc{RamosNabben2020TLS,
  author    = {Ramos, Luis García and Nabben, Reinhard},
  note      = {arXiv preprint, arXiv:2006.08750},
  title     = {A two-level shifted {L}aplace preconditioner for {H}elmholtz problems: Field-of-values analysis and wavenumber-independent convergence},
  year      = {2020},
  copyright = {arXiv.org perpetual, non-exclusive license},
  doi       = {10.48550/ARXIV.2006.08750},
  publisher = {arXiv},
}

@Article{RamosSeteEtAl2021PHE,
  author    = {Ramos, Luis García and Sète, Olivier and Nabben, Reinhard},
  journal   = {ETNA - Electronic Transactions on Numerical Analysis},
  title     = {Preconditioning the {H}elmholtz equation with the shifted {L}aplacian and {F}aber polynomials},
  year      = {2021},
  issn      = {1068-9613},
  pages     = {534--557},
  volume    = {54},
  doi       = {10.1553/etna_vol54s534},
  publisher = {Osterreichische Akademie der Wissenschaften},
}

@Book{Saad2003IMS,
  author    = {Yousef Saad},
  publisher = {Society for Industrial and Applied Mathematics},
  title     = {Iterative Methods for Sparse Linear Systems},
  year      = {2003},
  month     = {jan},
  doi       = {10.1137/1.9780898718003},
}

@Book{Saad2011NML,
  author    = {Yousef Saad},
  publisher = {Society for Industrial and Applied Mathematics},
  title     = {Numerical Methods for Large Eigenvalue Problems},
  year      = {2011},
  month     = {jan},
  doi       = {10.1137/1.9781611970739},
}

@Article{SaadSchultz1986GGM,
  author    = {Youcef Saad and Martin H. Schultz},
  journal   = {{SIAM} Journal on Scientific and Statistical Computing},
  title     = {{GMRES}: A Generalized Minimal Residual Algorithm for Solving Nonsymmetric Linear Systems},
  year      = {1986},
  month     = {jul},
  number    = {3},
  pages     = {856--869},
  volume    = {7},
  doi       = {10.1137/0907058},
  publisher = {Society for Industrial {\&} Applied Mathematics ({SIAM})},
}

@Article{Schot1992EYS,
  author    = {Schot, Steven H},
  journal   = {Historia Mathematica},
  title     = {Eighty years of {S}ommerfeld’s radiation condition},
  year      = {1992},
  issn      = {0315-0860},
  month     = nov,
  number    = {4},
  pages     = {385--401},
  volume    = {19},
  doi       = {10.1016/0315-0860(92)90004-u},
  publisher = {Elsevier BV},
}

@Article{Sommerfeld1912DGF,
  author  = {Sommerfeld, A.},
  journal = {Jahresbericht der Deutschen Mathematiker-Vereinigung},
  title   = {Die Greensche Funktion der Schwingungsgleichung},
  year    = {1912},
  pages   = {309-352},
  volume  = {21},
}

@Article{SpillaneSzyld2024NCA,
  author    = {Spillane, Nicole and Szyld, Daniel B.},
  journal   = {SIAM Journal on Matrix Analysis and Applications},
  title     = {New Convergence Analysis of {GMRES} with Weighted Norms, Preconditioning, and Deflation, Leading to a New Deflation Space},
  year      = {2024},
  issn      = {1095-7162},
  month     = oct,
  number    = {4},
  pages     = {1721--1745},
  volume    = {45},
  doi       = {10.1137/23m1622398},
  publisher = {Society for Industrial & Applied Mathematics (SIAM)},
}

@Article{Stefanov2000RRA,
  author    = {Stefanov, Plamen},
  journal   = {Comptes Rendus de l’Académie des Sciences - Series I - Mathematics},
  title     = {Resonances near the real axis imply existence of quasimodes},
  year      = {2000},
  issn      = {0764-4442},
  month     = jan,
  number    = {2},
  pages     = {105--108},
  volume    = {330},
  doi       = {10.1016/s0764-4442(00)00105-1},
  publisher = {Elsevier BV},
}

@Book{SzaboBabuska2021FEA,
  author    = {Szab{\'o}, Barna and Babu{\v{s}}ka, Ivo},
  title     = {Finite element analysis: Method, verification and validation},
  year      = {2021},
  publisher = {John Wiley \& Sons},
}

@Article{TangNabbenEtAl2009CTL,
  author    = {Tang, J. M. and Nabben, R. and Vuik, C. and Erlangga, Y. A.},
  journal   = {Journal of Scientific Computing},
  title     = {Comparison of Two-Level Preconditioners Derived from Deflation, Domain Decomposition and Multigrid Methods},
  year      = {2009},
  issn      = {1573-7691},
  month     = jan,
  number    = {3},
  pages     = {340--370},
  volume    = {39},
  doi       = {10.1007/s10915-009-9272-6},
  publisher = {Springer Science and Business Media LLC},
}

@Book{TrefethenEmbree2005SPB,
  author    = {Trefethen, Lloyd N. and Embree, Mark},
  publisher = {Princeton University Press},
  title     = {Spectra and Pseudospectra: The Behavior of Nonnormal Matrices and Operators},
  year      = {2005},
  isbn      = {9780691213101},
  month     = jan,
  doi       = {10.1515/9780691213101},
}

@Article{TurkelYefet1998APB,
  author    = {Turkel, E. and Yefet, A.},
  journal   = {Applied Numerical Mathematics},
  title     = {Absorbing {PML} boundary layers for wave-like equations},
  year      = {1998},
  issn      = {0168-9274},
  month     = aug,
  number    = {4},
  pages     = {533--557},
  volume    = {27},
  doi       = {10.1016/s0168-9274(98)00026-9},
  publisher = {Elsevier BV},
}

@Article{SluisVorst1986RCC,
  author    = {A. van der Sluis and H. A. van der Vorst},
  journal   = {Numerische Mathematik},
  title     = {The rate of convergence of Conjugate Gradients},
  year      = {1986},
  month     = {sep},
  number    = {5},
  pages     = {543--560},
  volume    = {48},
  doi       = {10.1007/bf01389450},
  publisher = {Springer Science and Business Media {LLC}},
}

@Article{Vuik2018KSS,
  author    = {Vuik, C.},
  journal   = {ESAIM: Proceedings and Surveys},
  title     = {Krylov Subspace Solvers and Preconditioners},
  year      = {2018},
  issn      = {2267-3059},
  pages     = {1--43},
  volume    = {63},
  doi       = {10.1051/proc/201863001},
  editor    = {Grigori, L. and Japhet, C. and Moireau, P.},
  publisher = {EDP Sciences},
}

@Article{VuikSegalEtAl2002CVD,
  author    = {Vuik, C. and Segal, A. and el Yaakoubi, L. and Dufour, E.},
  journal   = {Applied Numerical Mathematics},
  title     = {A comparison of various deflation vectors applied to elliptic problems with discontinuous coefficients},
  year      = {2002},
  issn      = {0168-9274},
  month     = apr,
  number    = {1},
  pages     = {219--233},
  volume    = {41},
  doi       = {10.1016/s0168-9274(01)00118-0},
  publisher = {Elsevier BV},
}

@Article{SimonciniSzyld2005OSC,
  author    = {Simoncini, Valeria and Szyld, Daniel B.},
  journal   = {{SIAM} Review},
  title     = {On the Occurrence of Superlinear Convergence of Exact and Inexact Krylov Subspace Methods},
  year      = {2005},
  issn      = {1095-7200},
  month     = 1,
  number    = {2},
  pages     = {247--272},
  volume    = {47},
  doi       = {10.1137/s0036144503424439},
  publisher = {Society for Industrial {\&} Applied Mathematics ({SIAM})},
}
